%% file: sasnr-plain.tex
\documentclass[a4paper, 11pt]{article}
\input{sasnr-head.tex}
\newcommand{\applink}{Finally, we provide proofs in \autoref{sec:proofs}.}

\begin{document}

\title{\thetitle
\footnotetext{\emph{Mathematics subject classification 2010.} \themsclass}
\footnotetext{\emph{Keywords.} \thekeywords}}
\author{\fullname\\\footnotesize \addressone \\\footnotesize \addresstwo \\
\footnotesize \theemail}
\date{}
\maketitle

\begin{abstract}
  \input{sasnr-abstract.tex}
\end{abstract}

\input{sasnr-body.tex}

\input{sasnr-proofs.tex}

\bibliographystyle{abbrvnat}
{\footnotesize \bibliography{sasnr}}

\end{document}

%% file: sasnr-head.tex
\newcommand{\thetitle}{Spatially-adaptive sensing in nonparametric
  regression}

\newcommand{\forenames}{Adam D.}
\newcommand{\surname}{Bull}
\newcommand{\fullname}{\forenames\ \surname}

\newcommand{\mscone}{62G08} 
\newcommand{\msctwo}{62L05} 
\newcommand{\mscthree}{62G20} 
\newcommand{\themsclass}{\mscone\ (primary); \msctwo, \mscthree\ (secondary).}

\newcommand{\kwdone}{Nonparametric regression}
\newcommand{\kwdtwo}{adaptive sensing}
\newcommand{\kwdthree}{sequential design}
\newcommand{\kwdfour}{active learning}
\newcommand{\kwdfive}{spatial adaptation}
\newcommand{\kwdsix}{spatially-inhomogeneous functions}
\newcommand{\thekeywords}{\kwdone, \kwdtwo, \kwdthree, \kwdfour,
  \kwdfive, \kwdsix.}

\newcommand{\addressone}{Statistical Laboratory}
\newcommand{\addresstwo}{University of Cambridge}

\newcommand{\theemail}{a.bull@statslab.cam.ac.uk}

\usepackage[round]{natbib}
\usepackage{amsmath, amsfonts, amssymb, amsthm, mathtools, thmtools, booktabs, 
tikz, mparhack, algorithm, algorithmic}
\usepackage[bookmarks, breaklinks, colorlinks, linkcolor=blue, citecolor=blue, 
pdfauthor={\fullname}, pdftitle={\thetitle}]{hyperref}

\newcommand{\comment}[1]{}
\let\oldmarginpar\marginpar
\renewcommand\marginpar[1]{\-\oldmarginpar[\raggedleft\footnotesize 
#1]{\raggedright\footnotesize #1}}
\newcommand{\acks}[1]{\subsection*{Acknowledgements}#1}

\DeclarePairedDelimiter{\abs}{\lvert}{\rvert}
\DeclarePairedDelimiter{\norm}{\lVert}{\rVert}
\newcommand{\N}{\mathbb{N}}
\newcommand{\Z}{\mathbb{Z}}

\newcommand{\R}{\mathbb{R}}
\renewcommand{\P}{\mathbb{P}}
\newcommand{\E}{\mathbb{E}}

\newcommand{\iid}{\overset{\mathrm{i.i.d.}}{\sim}}

\declaretheorem{theorem}
\declaretheorem[sibling=theorem]{lemma}

\declaretheorem[sibling=theorem]{definition}

\declaretheorem[sibling=theorem]{proposition}

\newcommand{\imin}{{i_n^{\min}}}
\newcommand{\imax}{{i_n^{\max}}}
\newcommand{\jmin}{{j_n^{\min}}}
\newcommand{\jmax}{{j_n^{\max}}}
\newcommand{\jmaxm}{{j_{n_m}^{\max}}}
\newcommand{\jmaxmm}{{j_{n_{m-1}}^{\max}}}

\DeclareMathOperator{\Bin}{Bin}
\DeclareMathOperator{\median}{median}


%% file: sasnr-abstract.tex
While adaptive sensing has provided improved rates of convergence in
sparse regression and classification, results in nonparametric
regression have so far been restricted to quite specific classes of
functions. In this paper, we describe an adaptive-sensing algorithm
which is applicable to general nonparametric-regression problems. The
algorithm is spatially-adaptive, and achieves improved rates of
convergence over spatially-inhomogeneous functions. Over standard
function classes, it likewise retains the spatial adaptivity
properties of a uniform design.


%% file: sasnr-body.tex
\section{Introduction}

In many statistical problems, such as classification and regression,
we observe data \(Y_1, Y_2, \dots,\) where the distribution of each
\(Y_n\) depends on a choice of design point \(x_n.\) Typically, we
assume the \(x_n\) are fixed in advance. In practice, however, it is
often possible to choose the design points sequentially, letting each
\(x_n\) be a function of the previous observations \(Y_1, \dots,
Y_{n-1}.\)

We will describe such procedures as {\em adaptive sensing}, but they
are also known by many other names, including sequential design,
adaptive sampling, active learning, and combinations thereof. The
field of adaptive sensing has seen much recent interest in the
literature: compared with a fixed design, adaptive sensing algorithms
have been shown to provide improvements in sparse regression
\citep{iwen_group_2009,haupt_distilled_2011,malloy_sequential_2011,boufounos_whats_2012,davenport_compressive_2012}
and classification
\citep{cohn_improving_1994,castro_minimax_2008,beygelzimer_importance_2009,koltchinskii_rademacher_2010,hanneke_rates_2011}.
Recent results have also focused on the limits of adaptive sensing
\citep{arias-castro_fundamental_2011,malloy_limits_2011,castro_adaptive_2012}.

In this paper, we will consider the problem of nonparametric
regression, where we aim to estimate an unknown function \(f : [0, 1]
\to \R\) from observations
\[Y_n \coloneqq f(x_n) + \varepsilon_n, \quad \varepsilon_n \iid N(0,
\sigma^2).\]  While previous authors have also considered this model
under adaptive sensing, their results have either been restricted to
quite specific classes of functions \(f,\) or have not provided
improved rates of convergence
\citep{faraway_sequential_1990,cohn_active_1996,hall_sequential_2003,castro_faster_2006,efromovich_optimal_2008}.

In the following, we will describe a new algorithm for adaptive
sensing in nonparametric regression. Our algorithm will be based on
standard wavelet techniques, but with an adaptive choice of design
points: we will aim to codify, in a meaningful way, the intuition that
we should place more design points in regions where \(f\) is hard to
estimate.

While many such heuristics are possible, we would like to construct an
algorithm with good theoretical justification; in particular, we will
be interested in attaining improved rates of convergence. In general,
however, it is known that in nonparametric regression, adaptive
sensing cannot provide improved rates over standard classes of
functions. \citet{castro_faster_2006} prove such a result for \(L^2\)
loss; we will show the same is true locally uniformly.

In the following, we will argue that the fault here lies not with
adaptive sensing, but rather with the functions considered. In the
field of {\em spatial adaptation}, unknown functions are often assumed
to be {\em spatially inhomogeneous}: they may be rougher, and thus
harder to estimate, in some regions of space than in others. The
seminal paper of \citet{donoho_ideal_1994} provides examples, which we
have reproduced in \autoref{fig:functions}; these mimic the kinds
functions observed in imaging, spectroscopy and other signal
processing problems.

\begin{figure}
\centering
\includegraphics[page=1,width=\textwidth]{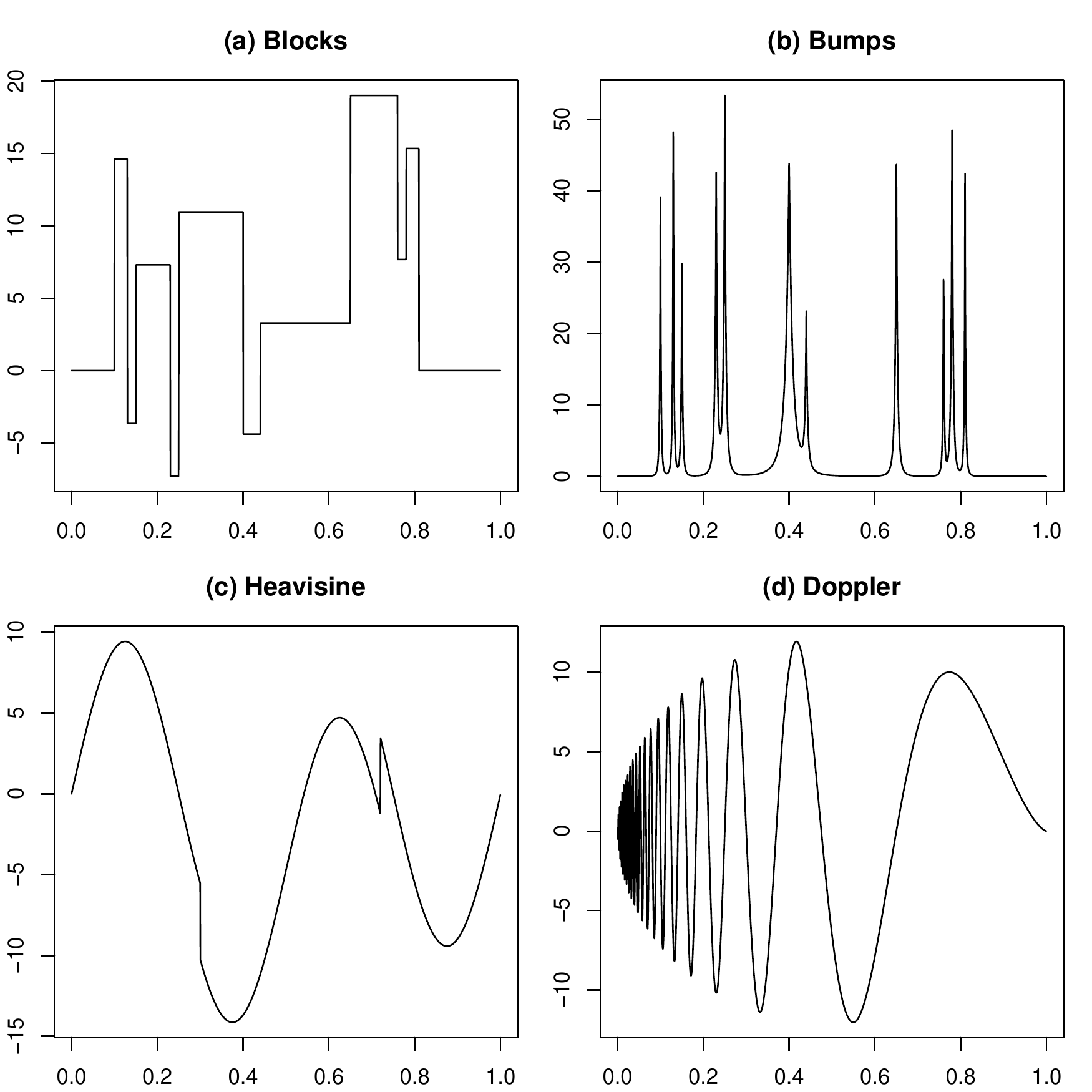}
\caption{Examples of spatially-inhomogeneous functions from
  \citet{donoho_ideal_1994}. Each function is scaled to have
  \(\mathrm{sd}(f) = 7.\)}
\label{fig:functions}
\end{figure}

Previous work in this field has provided many fixed-design estimators
with good performance over such functions
\citep{donoho_wavelet_1995,fan_data-driven_1995,lepski_optimal_1997,donoho_minimax_1998,fan_adaptation_1999}. With
adaptive sensing, however, we can obtain further improvements: if we
place more design points in regions where \(f\) is rough, our
estimates \(\hat f_n\) will become more accurate overall.

To quantify this, we will need to introduce new classes describing
spatially-inhomogeneous functions, over which our algorithm will be
shown to obtain improved rates of convergence. While these classes are
novel, they will be shown to contain quasi-all functions from standard
classes in the literature. Furthermore, our algorithm will be shown to
adaptively obtain near-optimal rates over both the new and standard
function classes.

Smoothness classes similar to our own have arisen in the study of
adaptive nonparametric inference
\citep{picard_adaptive_2000,gine_confidence_2010,bull_honest_2012},
and more generally also in the study of turbulence
\citep{frisch_singularity_1980,jaffard_frisch-parisi_2000}. As in
those papers, we find that for complex nonparametric problems, the
standard smoothness classes may be insufficient to describe behaviour
of interest; by specifying our target functions more carefully, we can
achieve more powerful results.

We might also compare this phenomenon to results in sparse regression,
where good rates are often dependent on specific assumptions about the
design matrix or unknown parameters
\citep{fan_sure_2008,van_de_geer_conditions_2009,meinshausen_stability_2010}. As
there, we can use the nature of our assumptions to provide insight
into the kinds of problems on which we can expect to perform well.

We will test our algorithm by estimating the functions in
\autoref{fig:functions} under Gaussian noise. We will see that, by
sensing adaptively, we can make significant improvements to accuracy;
we thus conclude that adaptive sensing can be of value in
nonparametric regression whenever the unknown function may be
spatially inhomogeneous.

In \autoref{sec:algorithm}, we describe our adaptive-sensing
algorithm.  In \autoref{sec:theory}, we describe our model of spatial
inhomogeneity, and show that adaptive sensing can lead to improved
performance over such functions. In \autoref{sec:practise}, we discuss
the implementation of our algorithm, and provide empirical
results. \applink

\section{The adaptive-sensing algorithm}
\label{sec:algorithm}

We now describe our adaptive-sensing algorithm in detail. We first
discuss how we estimate \(f\) under varying designs; we then move on
to the choice of design itself.

\subsection{Estimation under varying designs}
\label{sec:estimation}

Given observations \(Y_1, \dots, Y_n\) at a set of design points
\(\xi_n \coloneqq \{x_1, \dots, x_n\},\) we will estimate the function
\(f\) using the technique of wavelet thresholding, which is known to
give spatially-adaptive estimates \citep{donoho_ideal_1994}. To begin,
we will need to choose our wavelet basis; for \(j_0 \in \N,\) let
\[\varphi_{j,k} \text{ and } \psi_{j, 
  k},\quad j, k \in \Z,\ j \ge j_0,\ 0 \le k < 2^j,\]
be a compactly-supported wavelet basis of \(L^2([0, 1]),\) such as the
construction of \citet{cohen_wavelets_1993}.

In the following, we will assume the wavelets \(\psi_{j, k}\) have \(N
\in \N\) vanishing moments,
\[\int x^{n} \psi_{j, k}(x)\ dx = 0, \quad n \in \Z,\ 0 \le n < N,\]
and both \(\varphi_{j,k}\) and \(\psi_{j,k}\) are zero outside
intervals \(S_{j,k}\) of width \(2^{-j}(2L-1),\)
\[S_{j,k} \coloneqq 2^{-j}[k - L + 1, k + L) \cap [0, 1).\] For any
\(i \in \N,\ i \ge j_0,\) we may then write an unknown function \(f
\in L^2([0, 1])\) in terms of its wavelet expansion,
\[f = \sum_{k=0}^{2^{i}-1} \alpha_{i,k} \varphi_{i,k} +
\sum_{j=i}^\infty \sum_{k=0}^{2^j-1} \beta_{j, k} \psi_{j, k},\] and
estimate \(f\) in terms of the coefficients \(\alpha_{j_0, k},
\beta_{j,k}.\)

When the design is uniform, we can estimate these coefficients
efficiently in the standard way, using the fast wavelet transform
\citep{donoho_ideal_1994}. Suppose, as will always be the case in the
following, that the design points \(x_n\) are distinct, so we may
denote the observations \(Y_n\) as \(Y(x_n).\) Given \(i \in \N,\ i
\ge j_0,\) suppose also that we have observed \(f\) on a grid of
design points \(2^{-i}k,\ k \in \Z,\ 0 \le k < 2^i.\)

We may then estimate the scaling coefficients \(\alpha_{i,k}\) of \(f\) as
\[\hat \alpha^i_{i,k} \coloneqq 2^{-\frac{i}2}Y(2^{-i}k),\]
since for \(i\) large,
\begin{equation}
  \label{eq:scale-approx}
  \alpha_{i,k} = \int_{S_{i,k}} f(x) \varphi_{i,k}(x)\ dx \approx
  2^{-\frac{i}2}f(2^{-i}k).
\end{equation}
By an orthogonal change of basis, we can produce estimates \(\hat
\alpha_{j_0,k}^i\) and \(\hat \beta_{j,k}^i\) of the coefficients
\(\alpha_{j_0,k}\) and \(\beta_{j,k},\) given by the relationship
\begin{equation}
\label{eq:beta-hats-defn}
\sum_{k=0}^{2^{j_0}-1} \hat
\alpha^i_{j_0,k} \varphi_{j_0,k} + \sum_{j=j_0}^{i-1} \sum_{k=0}^{2^j-1}\hat
\beta^i_{j,k} \psi_{j,k} \coloneqq \sum_{k=0}^{2^i-1} \hat \alpha^i_{i,k} \varphi_{i,k}.
\end{equation}
These estimates can be computed efficiently by applying the fast
wavelet transform to the vector of values
\(2^{-\frac{i}2}Y(2^{-i}k).\)

Since we will be considering non-uniform designs, this situation will
often not apply directly. Many approaches to applying wavelets to
non-uniform designs have been considered in the literature, including
transformations of the data, and design-adapted wavelets
\citep[see][and references therein]{kerkyacharian_regression_2004}. In
the following, however, we will use a simple method, which allows us
to simultaneously control the accuracy of our estimates for many
different choices of design.

To proceed, we note that the value of an estimated coefficient \(\hat
\alpha^i_{j,k}\) or \(\hat \beta^i_{j,k}\) depends only on
observations \(Y(x)\) at points \(x \in S_{j,k} \cap 2^{-i}\Z.\) We
may therefore estimate the wavelet coefficients \(\alpha_{j_0,k}\) and
\(\beta_{j,k}\) by
\begin{equation}
\label{eq:alpha-beta-hat-defn}
\hat \alpha_{j_0,k} \coloneqq \hat \alpha_{j_0,k}^{i_n(j_0,k)}
\qquad \text{and} \qquad \hat \beta_{j,k} \coloneqq \hat
\beta_{j,k}^{i_n(j,k)},
\end{equation}
where the indices \(i_n(j,k)\) are chosen so that these estimates use
as many observations as possible,
\begin{equation}
\label{eq:i-defn}
i_n(j, k) \coloneqq \max\left\{i \in \N : i > j,\ 
  S_{j,k} \cap 2^{-i}\Z \subseteq \xi_n \right\}.
\end{equation}
To ease notation, for now we will estimate coefficients only up to a
maximum resolution level \(\jmax \in \N,\) \(\jmax > j_0,\) chosen so
that \(2^\jmax \sim n/\log(n).\) We will then be able to guarantee
that the set in \eqref{eq:i-defn} is non-empty.

Using these estimates directly will lead to a consistent estimate of
\(f,\) but one converging very slowly; to obtain a spatially-adaptive
estimate, we must use thresholding. We fix \(\kappa > 1,\) and for
\begin{equation}
\label{eq:e-defn}
e_n(j, k) \coloneqq \sigma 2^{-\frac12 i_n(j, k)}\sqrt{2 \log(n)},
\end{equation}
define the hard-threshold estimates
\[\hat{\beta}_{j,k}^T \coloneqq \begin{cases} \hat{\beta}_{j,k}, & 
  \abs{\hat{\beta}_{j,k}} \ge \kappa e_n(j, k),\\
  0, & \mathrm{otherwise}.\end{cases}\]%
We then estimate \(f\) by
\begin{equation}
\label{eq:f-hat-defn}
\hat{f}_n \coloneqq \sum_{k=0}^{2^{j_0}-1} \hat \alpha_{j_0,k}
\varphi_{j_0,k} + \sum_{j=j_0}^{\jmax-1} \sum_{k=0}^{2^j-1}
\hat{\beta}_{j,k}^T \psi_{j,k}.
\end{equation}
Given a uniform design \(\xi_n = 2^{-i}\Z \cap [0, 1),\) this is a
standard hard-threshold estimate; otherwise it gives a generalisation
to non-uniform designs.

\subsection{Adaptive design choices}

So far, we have only discussed how to estimate \(f\) from a fixed
design.  However, we can also use these estimates to choose the design
points adaptively. We will choose the design points in stages, at
stage \(m\) selecting points \(x_{n_{m-1}+1}\) to \(x_{n_m}\) in terms
of previous observations \(Y_1, \dots, Y_{n_{m-1}}.\) The number of
design points in each stage can be chosen freely, subject only to the
conditions that \(n_0\) is a power of two, and the ratios
\(n_m/n_{m-1}\) are bounded away from 1 and \(\infty.\) We may, for
example, choose
\begin{equation}
\label{eq:n-defn}
n_m \coloneqq \lfloor 2^{j+\tau m} \rfloor,
\end{equation}
for some \(j \in \N,\) and \(\tau > 0.\)

In the initial stage, we will choose \(n_0\) design points spaced
uniformly on \([0, 1],\)
\[x_i \coloneqq (i-1)/n_0,\quad 1 \le i \le n_0.\]%
At further stages \(m \in \N,\) we will construct a {\em target
  density} \(p_m\) on \([0, 1],\) and then select design points
\(x_{n_{m-1}+1}, \dots, x_{n_m}\) so that the design \(\xi_{n_m}\)
approaches a draw from this density. We will choose \(p_m\) to be
concentrated in regions of \([0,1]\) where we believe the function
\(f\) is difficult to estimate, ensuring that later design points will
adapt to the unknown shape of \(f.\)

At time \(n_{m-1},\) for each \(j \in \N,\ j_0 \le j < \jmaxmm,\) we
rank the \(2^j\) thresholded estimates \(\hat{\beta}^T_{j, k}\) in
decreasing order of size. We then have
\[\abs*{\hat{\beta}^T_{j, r_j^{-1}(1)}} \ge \dots \ge \abs*{\hat{\beta}^T_{j, 
    r_j^{-1}(2^j)}},\]%
for a bijective ranking function \(r_j.\) We will choose the target
density so that, in the support of each significant term
\(\beta_{j,k}\psi_{j, k}\) in the wavelet series, the density will be,
up to log factors, at least \(2^j/r_j(k).\) The density will thus be
concentrated in regions where the wavelet coefficients are known to be
large. To ensure that all coefficients are estimated accurately, we
will also require the density to be bounded below by a fixed constant,
given by a choice of parameter \(\lambda > 0.\)

Split the interval \([0, 1]\) into sub-intervals \[I_{l,m} \coloneqq
2^{-\jmaxm}[l, l+1),\quad l \in \Z,\ 0 \le l < 2^\jmaxm.\] We define
the target density on \(I_{l,m}\) to be
\begin{multline*}
p_{l,m} \coloneqq A \max\bigg(\left\{\lambda\right\} \cup\\
\left\{\frac{2^j}{r_j(k)(\jmaxmm)^2} :
 j \in \N,\ j_0 \le j < \jmaxmm,\ I_{l,m} \subseteq S_{j,k},\ \hat{\beta}^T_{j,k} \ne
    0\right\}\bigg),
\end{multline*}
where the fixed constant \(A > 0\) is chosen so that the density
\(p_m\) always integrates to at most one,
\(2^{-\jmaxm} \sum_l p_{l,m} \le 1.\)
The specific value of \(A\) is unimportant, but note that
\[
  2^{-\jmaxm} \sum_{l=0}^{2^\jmaxm-1} p_{l,m} \lesssim
  1 + (\jmaxmm)^{-2}\sum_{j=0}^\jmaxmm \sum_{k=1}^{2^j} k^{-1}
  \lesssim 1,\]
so such a choice of \(A\) exists.

We now aim to choose new design points \(x_{n_{m-1}+1},\dots,x_{n_m}\)
so that the design \(\xi_n\) approximates a draw from \(p_m.\) To
simplify notation, we first include any points \(x \in 2^{-\jmaxm}\Z
\cap [0, 1)\) not already in the design. We will assume the \(n_m\)
and \(\jmax\) are chosen so that this requires no more than \(n_m -
n_{m-1}\) design points; since \(\jmax\) is defined only
asymptotically, and \(2^{\jmaxm} = o(n_m-n_{m-1}),\) such a choice is
always possible.

We then construct an {\em effective density} \(q_{m,n},\) describing a
nominal density generating the design \(\xi_n.\) This density will be
at least \(2^i/n\) on any region where the design contains the grid
\(2^{-i}\Z;\) it will thus describe the density of all design points
on regular grids. We define the effective density on \(I_{l,m}\) at
time \(n\) to be
\[q_{l,m,n} \coloneqq n^{-1} \max\left\{2^i : i \in \N, \ 2^{-i}\Z
  \cap I_{l,m} \subseteq \xi_n \right\}.\]%
Again, note this density integrates to at most one,
\(2^{-\jmaxm}\sum_l q_{l,m,n} \le 1.\)

Our remaining goal is to choose the new design points so that the
effective density approaches the target density. In our proofs, we
will require control over the maximum discrepancy from \(p_m\) to
\(q_{m,n},\)
\begin{equation}
\label{eq:max-disc}
\max_{l=0}^{2^\jmaxm-1} p_{l,m}/q_{l,m,n}.
\end{equation}
To choose the next stage of design points, having selected points
\(x_1, \dots, x_n,\) we therefore pick an \(l\) maximising
\eqref{eq:max-disc}; note that this does not require us to calculate
\(A.\) We then add points \(2^{-i}\Z \cap I_{l,m}\) to the design,
choosing the smallest index \(i\) for which at least one such point is
not already present.

In doing so, we halve the largest value of \(p_{l,m}/n q_{l,m,n},\)
while leaving all other such values unchanged. Repeating this process,
we will therefore add design points on grids \(2^{-i}\Z \cap I_{l,m}\)
so as to minimise \eqref{eq:max-disc}. We continue until we have
selected a total of \(n_m\) design points; for convenience, let
\(q_{l,m} \coloneqq q_{l,m,n_m}\) denote the effective density on
\(I_{l,m}\) once we are done.

The final algorithm is thus described by \autoref{alg:sas}; it can be
implemented efficiently using a priority queue to find values of \(l\)
maximising \eqref{eq:max-disc}. We will show that this algorithm
ensures the final effective density \(q_m\) is close to the target
density \(p_m,\) and that estimates made under it are therefore
spatially-adaptive for a wide variety of functions.

\begin{algorithm}
  \caption{Spatially-adaptive sensing}
  \label{alg:sas}
\begin{algorithmic}
  \STATE \(n \gets n_0\)
  \STATE \(x_1, \dots, x_n \gets n^{-1}\Z \cap [0, 1)\)
  \STATE observe \(Y_1, \dots, Y_n\)
  \STATE \(m \gets 1\)
  \LOOP
  \STATE \(x_{n+1}, \dots, x_{n'} \gets 2^{-\jmaxm}\Z \cap [0, 1)
  \setminus \xi_n\)
  \STATE \(n \gets n'\)
  \WHILE{\(n < n_m\)}
  \STATE choose \(l\) maximising \(p_{l,m}/q_{l,m,n}\)
  \STATE \(S \gets 2^{-i}\Z \cap I_{l,m} \setminus \xi_n,\) for the
  smallest \(i\) such that \(S \ne \emptyset\)
  \REPEAT
  \STATE \(n \gets n + 1\)
  \STATE choose \(x_n \in S\)
  \STATE \(S \gets S \setminus \{x_n\}\)
  \UNTIL \(S = \emptyset\) or \(n = n_m\)
  \ENDWHILE
  \STATE observe \(Y_{n_{m-1}+1}, \dots, Y_n\)
  \STATE estimate \(f\) by \(\hat f_n\)
  \STATE \(m \gets m + 1\)
  \ENDLOOP
\end{algorithmic}
\end{algorithm}

\section{Theoretical results}
\label{sec:theory}

We now provide theoretical results on the performance of our
algorithm. We begin by defining the relevant function classes, then
discuss our choice of functions considered; we conclude with our
results on convergence rates.

\subsection{Function classes}
\label{sec:classes}

We first define the function classes we will consider in the
following. We will assume we have a wavelet basis
\(\psi_{j_0,k},\varphi_{j,k}\) satisfying the assumptions of
\autoref{sec:estimation}; we can then describe any function \(f \in
L^2([0, 1])\) by its wavelet series,
\[f = \sum_{k=0}^{2^{j_0}-1} \alpha_{j_0,k} \varphi_{j_0,k} +
\sum_{j=j_0}^\infty \sum_{k=0}^{2^j-1} \beta_{j, k} \psi_{j, k}.\]

The smoothness of \(f,\) and thus the ease with which it can be
estimated, is determined by the size of the coefficients
\(\alpha_{j_0,k},\beta_{j,k};\) \(f\) is smooth, and easily estimated,
when these coefficients are small. The smoothness of a function can be
described in terms of its membership of standard function classes.
While there are many such classes, in what follows we will be
interested primarily in the H\"older and Besov scales
\citep{hardle_wavelets_1998}.

For \(s \in \N,\) the H\"older classes \(C^s(M)\) contain all
functions which are \(s\)-times differentiable, with value and
derivatives are bounded by \(M;\) the local H\"older classes \(C^s(M,
I)\) instead require this condition to hold only over an interval
\(I.\) These definitions can also be extended to non-integer \(s,\)
and given in terms of wavelet coefficients. We note that while the
wavelet definitions are in general slightly weaker than the classical
ones, this will not fundamentally affect our results in what follows.

\begin{definition}
  For \(s \in (0, N),\) \(M > 0,\) and \(I \subseteq [0, 1],\)
  \(C^s(M, I)\) is the class of functions \(f \in L^2([0, 1])\)
  satisfying
  \[\max\left(2^{j_0\left(s+\frac12\right)}\sup_{k:S_{j_0,k} \subseteq I}\abs{\alpha_{j_0,k}},\
    \sup_{j=j_0}^\infty 2^{j\left(s+\frac12\right)} \sup_{k:S_{j,k}
      \subseteq I} \abs{\beta_{j,k}}\right) \le M.\] For \(I =
  [0,1],\) we denote this class \(C^s(M).\)
\end{definition}

The Besov classes \(B^r_{p,\infty}(M)\) are more general. For \(p =
\infty,\) they coincide with our definition of the H\"older classes
\(C^r(M).\) For \(p < \infty,\) they allow functions with some
singularities, provided they are still, on average, \(r\)-times
differentiable; smaller values of \(p\) correspond to more irregular
functions.

\begin{definition}
  For \(r \in (0, N),\) \(p \in [1, \infty),\) and \(M > 0,\)
  \(B^r_{p, \infty}(M)\) is the class of functions \(f \in L^2([0,
  1])\) satisfying
  \[\max\left(2^{j_0\left(r+\frac12-\frac1p\right)} \left(\sum_k\abs{\alpha_{j_0,k}}^p\right)^{\frac1p},\ 
    \sup_{j=j_0}^\infty 2^{j\left(r+\frac12-\frac1p\right)} \left(
      \sum_k \abs{\beta_{j,k}}^p \right)^{\frac1p}\right) \le M.\] For
  \(p = \infty,\) we define \(B^r_{\infty, \infty}(M) \coloneqq C^r(M).\)
\end{definition}

Many other standard classes are related to these Besov classes,
including the Sobolev classes \(W^{r,p}(M) \subseteq B^r_{p,\infty}(M),\)
the Sobolev Hilbert classes \(H^r(M) \subseteq B^r_{2,\infty}(M),\) and
the functions of bounded variation \(BV(M) \subseteq
B^1_{1,\infty}(M).\) In each case, convergence rates are unchanged by
considering the containing Besov class, meaning we need consider only
Besov classes in what follows.

Besov classes can also be thought of as describing functions whose
wavelet expansions are {\em sparse}. From the above definitions, we
can see that, compared to a H\"older class, functions in a Besov class
can have a number of larger wavelet coefficients, provided there are
not too many. In other words, functions in a Besov class can have
wavelet expansions where most, but not all, coefficients are small.

Besov classes are often used to describe spatially-inhomogeneous
functions; we can see why by considering \autoref{fig:wavelets}, which
plots the wavelet coefficients of the functions in
\autoref{fig:functions}. As above, the coefficients are often, but not
always, small.

\begin{figure}
  \centering
\includegraphics[page=2,width=\textwidth]{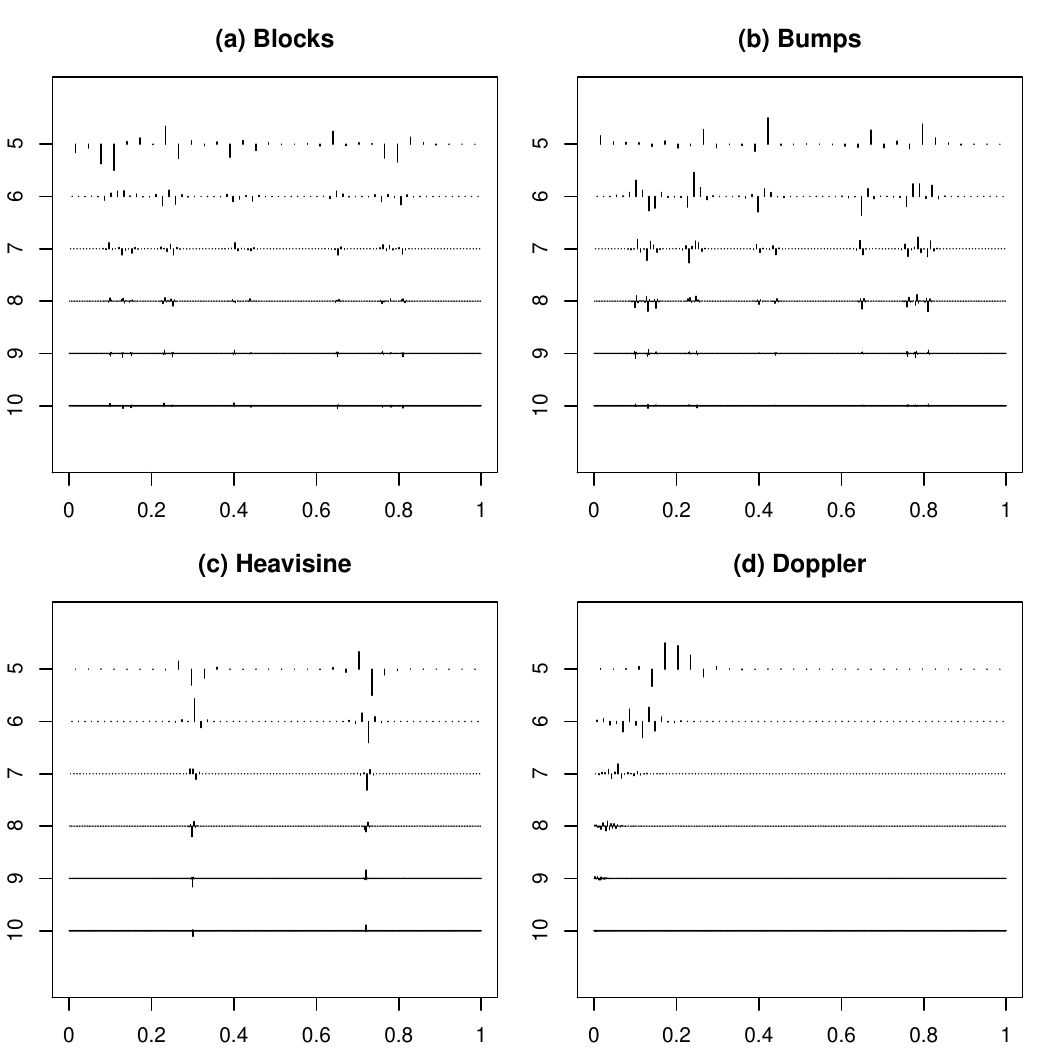}
\caption{CDV-8 wavelet coefficients for the functions in
  \autoref{fig:functions}. The height of each line corresponds to a
  wavelet coefficient \(\beta_{j,k};\) the \(x\)-axis plots the
  location \(2^{-j}k,\) and the \(y\)-axis the scale \(j.\)}
\label{fig:wavelets}
\end{figure}

Our final function class is a new definition, which we will argue
captures another typical feature of spatially-inhomogeneous functions,
and is necessary to obtain improved rates of convergence. From
\autoref{fig:wavelets}, we can see that, in regions where the
functions \(f\) are rough, their wavelet coefficients are often large;
in regions where they are smooth, their coefficients are small.  In
other words, if \(f\) is difficult to approximate in some region at
high resolution, it will also be difficult to approximate there at
lower resolutions.

We will call such functions {\em detectable}, and describe them in
terms of detectable classes \(D^s_t(M, I) \subseteq C^s(M, I).\) The
additional parameter \(t \in (0, 1)\) controls the strength of our
condition; larger \(t\) corresponds to a stronger condition on
functions \(f.\)

\begin{definition}
  For \(s \in (0, N),\) \(t \in (0, 1),\) \(M > 0,\) and an interval
  \(I \subseteq [0, 1],\) \(D^s_{t}(M, I)\) is the class of functions
  \(f \in C^s(M, I)\) which also satisfy
  \begin{multline}
    \label{eq:parent-cond}
    \forall\ j \in \N,\ j \ge \lceil \tfrac{j_0}t \rceil,\ k : S_{j,k} \cap I \ne \emptyset,\\
    \exists\ j' \in \N,\ \lfloor tj \rfloor \le j' < j,\ k' : S_{j',k'}
    \supset S_{j,k},\\
    \abs{\beta_{j',k'}} \ge
    (j'/j)2^{(j-j')\left(s+\frac12\right)}\abs{\beta_{j, k}}.
  \end{multline}
\end{definition}

The definition thus requires that each term in the wavelet series on
\(I,\) at a fine scale \(j,\) lies within the support of another term,
of comparable size, at a coarser scale \(j'.\) The parameter \(s\)
controls how large this second term must be, and \(t\) controls how
far apart the scales \(j\) and \(j'\) can be.

In \autoref{sec:choice}, we will discuss why such conditions may be
natural to consider for this problem. First, however, we will
establish that a typical locally H\"older function will be detectable;
indeed, similarly to results in \citet{jaffard_frisch-parisi_2000} and
\citet{gine_confidence_2010}, we can show that the set of functions
which are locally H\"older but not detectable is topologically
negligible. We may therefore sensibly restrict to detectable functions
in what follows.

\begin{proposition}
  \label{prop:negligible-set}
  For \(s \in (0, N),\) \(M > 0,\) and any
  interval \(I \subseteq [0, 1],\) define
  \[\mathcal D \coloneqq C^s(M, I) \setminus \bigcup_{t \in (0, 
    1)} D^s_{t}(M, I).\] Then \(\mathcal D\) is nowhere dense in the
  norm topology of \(C^s(M, I).\)
\end{proposition}

\subsection{Spatially-inhomogeneous functions}
\label{sec:choice}

We now discuss our choice of functions to consider. We begin with some
well-known results, which describe the limitations of adaptive sensing
over H\"older classes. Let
\[\alpha(s) \coloneqq s/(2s+1),\]
and define an adaptive-sensing algorithm to be a choice of design
points \(x_n = x_n(Y_1, \dots, Y_{n-1}),\) together with an estimator
\(\hat f_n = \hat f_n(Y_1, \dots, Y_n).\) Then, up to log
factors, a spatially-adaptive estimate can attain the rate
\(n^{-\alpha(s)}\) over any local H\"older class \(C^s(M, I),\) and
this cannot be improved upon by adaptive sensing.

\begin{theorem}
  \label{thm:holder-rates}
  Using a uniform design \(x_i = (i-1)/n,\) there exists an estimator
  \(\hat f_n,\) which satisfies
  \[\sup_{x \in J} \abs{\hat f_n(x) - f(x)} = O_p(c_n),\]
  uniformly over \(f \in C^s(M, I) \cap C^{\frac12}(M),\) for any \(s
  \in [\frac12, N),\) \(M > 0,\) \(I\) an interval open in \([0,
  1],\) \(J \subseteq I\) a closed interval, and
  \[c_n = (n/\log(n))^{-\alpha(s)}.\]
\end{theorem}

\begin{theorem}
  \label{thm:holder-lower-bound}
  Let \(s \ge \frac12,\) \(M > 0,\) \(I\) an interval open in \([0,
  1],\) and \(J \subseteq I\) a closed interval. Given an
  adaptive-sensing algorithm with estimator \(\hat f_n,\) if
  \[\sup_{x \in J} \abs{\hat f_n(x) - f(x)} = O_p(c_n)\]
  uniformly over \(f \in C^s(M, I) \cap C^{\frac12}(M),\) then
  \[c_n \gtrsim n^{-\alpha(s)}.\]
\end{theorem}

To benefit from adaptive sensing, we will need to exploit two features
of the functions in \autoref{fig:functions}. The first is that, as
discussed in \autoref{sec:classes}, these functions are sparse: they
are rougher in some regions than others. This sparsity is necessary to
benefit from adaptive sensing: it is the difference between rough and
smooth which allows us to improve performance, placing more design
points in rougher regions.

Sparsity is commonly measured in terms of Besov, rather than H\"older,
classes. This change alone, however, is not enough to allow us to
benefit from adaptive sensing. Since \(B^r_{p,\infty}(M) \subseteq
C^{r-1/p}(M),\) over this class we can achieve the rate
\(n^{-\alpha\left(r-1/p\right)},\) up to log factors, with the
fixed-design method of \autoref{thm:holder-rates}. We can further show
that, in this case, adaptive sensing offers little improvement.

\begin{theorem}
  \label{thm:besov-lower-bound}
  Let \(p \in [1, \infty],\) \(r \ge \frac12+\frac1p,\) \(M > 0,\) and
  \(I\) be an interval in \([0, 1].\) Given an adaptive-sensing algorithm
  with estimator \(\hat f_n,\) if
  \[\sup_{x \in I}\abs{\hat f_n(x) -f(x)} = O_p(c_n),\]
  uniformly over \(f \in B^r_{p,\infty}(M),\) then
  \[c_n \gtrsim n^{-\alpha\left(r-\frac1p\right)}.\]
\end{theorem}  

To benefit from adaptive sensing, it is not enough to have regions in
which the function is rough or smooth; we must also be able to detect
where those regions are. This is the rationale behind our detectable
classes \(D^s_t(M, I)\): if a function is detectable, its roughness at
high resolutions will be signalled by corresponding roughness at low
resolutions, which we can observe in advance.

We will be interested in functions \(f\) which are both sparse and
detectable.  For \(p \in [1, \infty],\) \(r \in [\frac12+\frac1p,
N),\) \(s \in [r-\frac1p, N),\) \(t \in (0,1),\) \(M > 0,\) and any
interval \(I \subseteq [0, 1],\) let
\begin{equation}
\label{eq:f-defn}
\mathcal F = \mathcal F(p, r, s, t, M, I) \coloneqq B^r_{p,\infty}(M) \cap D^{s}_{t}(M,I)
\end{equation}
denote a class of sparse and detectable functions.

We note that this class has two smoothness parameters: \(r\) governs
the average global smoothness of a function \(f \in \mathcal F,\)
while \(s\) governs its local smoothness over \(I.\) Since functions
in \(B^r_{p,\infty}\) are everywhere at least \((r-\frac1p)\)-smooth,
we have restricted to the interesting case \(s \ge r - \frac1p.\)

From \autoref{prop:negligible-set}, we know that quasi-all locally
H\"older functions are detectable; we can likewise show that under a
fixed design, restricting to sparse and detectable functions does not
alter the minimax rate of estimation. We may thus conclude that
requiring sparsity and detectability thus does not make estimation
fundamentally easier.

\begin{theorem}
  \label{thm:fixed-lower-bound}
  Using a fixed design, if an estimator \(\hat f_n\) satisfies
  \[\sup_{x \in I} \abs{\hat f_n(x) - f(x)} = O_p(c_n),\]
  uniformly over \(f \in \mathcal F,\) then
  \[c_n \gtrsim (n/\log(n))^{-\alpha(s)}.\]
\end{theorem}

\subsection{Benefits of adaptive sensing}

With adaptive sensing, however, we can take advantage of these
conditions to obtain improved rates of convergence. We even can show
that \autoref{alg:sas} achieves this without knowledge of the class
\(\mathcal F;\) we can thus adapt not only to the regions where \(f\)
is rough, but also to the overall smoothness and sparsity of \(f.\)

\begin{theorem}
  \label{thm:alg-rates}
  \autoref{alg:sas} satisfies
  \[\sup_{x \in I} \abs{\hat f_n(x) - f(x)} = O_p(c_n),\]
  uniformly over \(f \in \mathcal F,\) for \(u \coloneqq \max(r-s,
  0),\)
  \begin{equation}
    \label{eq:rs-prime-defn}
    r' \coloneqq s + tu, \qquad s' \coloneqq s/(1-ptu),
  \end{equation}
  and
  \begin{equation}
    \label{eq:alg-rates}
    c_n \coloneqq \left(n/\log(n)^3\right)^{-\alpha\left(\min(r', s')\right)}\log(n)^{1(r'=s')}.
  \end{equation}
\end{theorem}

We thus obtain, up to log factors, the weaker of the two rates
\(n^{-\alpha(r')}\) and \(n^{-\alpha(s')}.\) Both of these rates are
at least as good as the \(n^{-\alpha(s)}\) bound faced by a fixed
design; when \(s < r,\) and the function \(f\) may be locally rough,
the rates are strictly better. In that case, we obtain the
\(n^{-\alpha(r')}\) rate when \(s/r \ge (1-t)/(2-t),\) and the
\(n^{-\alpha(s')}\) rate otherwise.

The improvement is driven by two parameters: \(t,\) which governs how
easy it is to detect irregularities of \(f,\) and \(u,\) which governs
how much rougher \(f\) is locally than on average. When both \(t\) and
\(u\) are large, the rates we obtain are significantly improved; in
the most favourable case, when \(u = 1,\) and \(t \approx 1,\) this
result is equivalent to gaining an extra derivative of \(f.\) We can
even show that these rates are near-optimal over classes \(\mathcal
F.\)

\begin{theorem}
  \label{thm:adaptive-lower-bound}
  Given an adaptive-sensing algorithm with estimator \(\hat f_n,\) if
  \[\sup_{x \in I}\abs{\hat f_n(x) -f(x)} = O_p(c_n),\]
  uniformly over \(f \in \mathcal F,\) then
  \[c_n \gtrsim n^{-\alpha(\min(r', s'))},\]
  for \(r', s'\) given by \eqref{eq:rs-prime-defn}.
\end{theorem}

Furthermore, we also have that, even in the absence of sparsity or
detectability, we still achieve the spatial adaptation properties of a
fixed design. We may thus use our adaptive-sensing algorithm with the
confidence that, even if \(f\) is spatially homogeneous, we will not
pay an asymptotic penalty.

\begin{theorem}
  \label{thm:alg-holder-rates}
  \autoref{alg:sas} satisfies
  \[\sup_{x \in J}\abs{\hat f_n(x) - f(x)} = O_p\left(c_n\right),\]
  uniformly over \(f \in C^s(M, I) \cap C^{\frac12}(M),\) for any \(s
  \in [\frac12, N),\) \(M > 0,\) \(I\) an interval open in \([0, 1],\)
  \(J \subseteq I\) a closed interval, and
  \[c_n \coloneqq \left(n/\log(n)\right)^{-\alpha(s)}.\]
\end{theorem}

\section{Implementation and experiments}
\label{sec:practise}

We now give some implementation details of \autoref{alg:sas}, and
provide empirical results. Before we test the algorithm, we must
describe how we compute \(\hat f_n,\) and choose the parameters
governing the algorithm's behaviour.

\subsection{Estimating functions}

For simplicity, in \eqref{eq:f-hat-defn} we defined \(\hat f_n\) in
terms of wavelets only up to the resolution level \(\jmax.\) While
asymptotically this carries no penalty, in finite time we may do
better by estimating all the wavelets for which we have available
data. In other words, we use the estimate
\begin{equation}
\label{eq:new-f-hat-defn}
\hat{f}_n \coloneqq \sum_{k=0}^{2^{j_0}-1} \hat \alpha_{j_0,k}
\varphi_{j_0,k} + \sum_{j=j_0}^{\infty} \sum_{k=0}^{2^j-1}
\hat{\beta}_{j,k}^T \psi_{j,k},
\end{equation}
where for \(j \ge \jmax,\) if the set in \eqref{eq:i-defn} is empty,
we let \(i_n(j,k) \coloneqq -\infty,\) forcing \(\hat \beta_{j,k}^T =
0.\) We note that since there are finitely many design points, the sum
in \eqref{eq:new-f-hat-defn} must have finitely many non-zero terms.

To compute these estimates \(\hat f_n,\) we must convert the estimated
coefficients \(\hat \alpha_{j_0,k},\) \(\hat \beta_{j,k}^T\) back into
function values \(\hat f_n(x).\) For \(i \in \N,\ i \ge j_0,\) to
evaluate \(\hat f_n\) at points \(x = 2^{-i}k,\) \(k \in \Z,\ 0 \le k
< 2^i,\) we make the approximation
\begin{equation}
\label{eq:rev-scale-approx}
\hat f_n(2^{-i}k) \approx 2^{\frac{i}2}\hat \alpha^T_{i,k},
\end{equation}
where the post-thresholding scaling coefficients \(\hat
\alpha^T_{i,k}\) are defined by
\[\sum_{k=0}^{2^i-1}\hat \alpha^T_{i,k}\varphi_{i,k} \coloneqq \hat f_n.\]
These can again be computed efficiently using the fast wavelet
transform.

Given a uniform design, and predicting \(f\) only at the design
points, this is enough to give estimates \(\hat f_n;\) if we set
\(\kappa = 1,\) we have just described a standard hard-threshold
wavelet estimate \citep{donoho_ideal_1994}. In that case, the
observations and predictions are always made at the same scale,
\(i_n(j, k) = i,\) so the errors in \eqref{eq:scale-approx} and
\eqref{eq:rev-scale-approx} tend to cancel out. In other cases,
however, the observations and predictions may be at different scales;
these errors then may build up, making \(\hat f_n\) look like a
translation of \(f.\)

To resolve the issue, we will use a slightly different definition of
the estimated coefficients \(\hat \alpha_{j_0,k}\) and \(\hat
\beta_{j,k},\) which ensures the scales of observation and prediction
are the same. Given \(i \in \N,\ i \ge j_0,\) to estimate \(f\) at
points \(x = 2^{-i}k,\) \(k \in \Z,\ 0 \le k < 2^i,\) we set
\(x_{n,k} \coloneqq \sup\{x \in \xi_n : x \le 2^{-i}k\},\) and let
\[\hat \alpha_{i,k} \coloneqq 2^{-\frac{i}2} Y(x_{n,k}).\]%
We then define the estimates \(\hat \alpha_{j_0,k}\) and \(\hat
\beta_{j,k}\) by
\[\sum_{k=0}^{2^{j_0}-1} \hat
\alpha_{j_0,k} \varphi_{j_0,k} + \sum_{j=j_0}^{i-1}
\sum_{k=0}^{2^j-1}\hat \beta_{j,k} \psi_{j,k} \coloneqq
\sum_{k=0}^{2^i-1} \hat \alpha_{i,k} \varphi_{i,k},\] using the fast
wavelet transform as before.

We note that this definition is approximately the same as the one in
\eqref{eq:alpha-beta-hat-defn}; while it is harder to control
theoretically, it gives improved experimental behaviour. We also note
that, with a uniform design, if we wish to predict \(f\) only at the
design points, this again reduces to a standard wavelet estimate.

\subsection{Choosing parameters}
\label{sec:parameters}

To apply \autoref{alg:sas}, we must choose the parameters \(\kappa,\)
\(\lambda\) and \(\tau,\) and also estimate \(\sigma\) if it is not
already known. The parameter \(\kappa\) governs the size of our
wavelet thresholds: larger \(\kappa\) means we will be more
conservative. While our theoretical results are proved for choices
\(\kappa > 1,\) in our empirical tests we took \(\kappa = 1.\) This
gives a simple choice of threshold which performs well, and allows us
to compare our results with standard hard-threshold estimates.

The parameter \(\lambda\) controls how uniform we make our design
points: for \(\lambda \gg 0\) the design points will be mostly
uniform, while for \(\lambda \approx 0\) they will be concentrated at
irregularities of \(f.\) The parameter \(\tau\) likewise controls how
many design points we choose at each stage: for \(\tau \gg 0\) there
will be a few large stages, while for \(\tau \approx 0\) there will be
many small ones. Empirically, we found the values \(\lambda = \tau =
\frac12\) gave good trade-offs.

Finally, for uniform designs, \citet{donoho_ideal_1994} suggest
estimating \(\sigma\) by the median size of the \(\hat \beta_{j,k}\)
at fine resolution scales. Our designs may not be uniform, but they
are guaranteed to provide us with estimates \(\hat \beta_{j,k}\) up to
level \(\jmax-1.\) We will therefore use the similar estimate
\[\hat \sigma_n \coloneqq \median\{2^{\frac12i}\abs{\hat \beta_{j,k}} : j \ge
\jmax-1, i_n(j,k) > -\infty\}/0.6745,\] which includes all estimated
coefficients at scales at least this fine.

\subsection{Empirical results}
\label{sec:data}

We now describe the results of using \autoref{alg:sas} to estimate the
functions in \autoref{fig:functions}, observing under \(N(0,
\sigma^2)\) noise. \autoref{fig:noise} plots \(n = 2^{11}\) noisy
samples of each function, under a uniform design, with \(\sigma = 1,\) while
\autoref{fig:fixed} plots a standard wavelet threshold estimate from
these samples.

\begin{figure}[p]
\centering
\includegraphics[width=\textwidth,page=3]{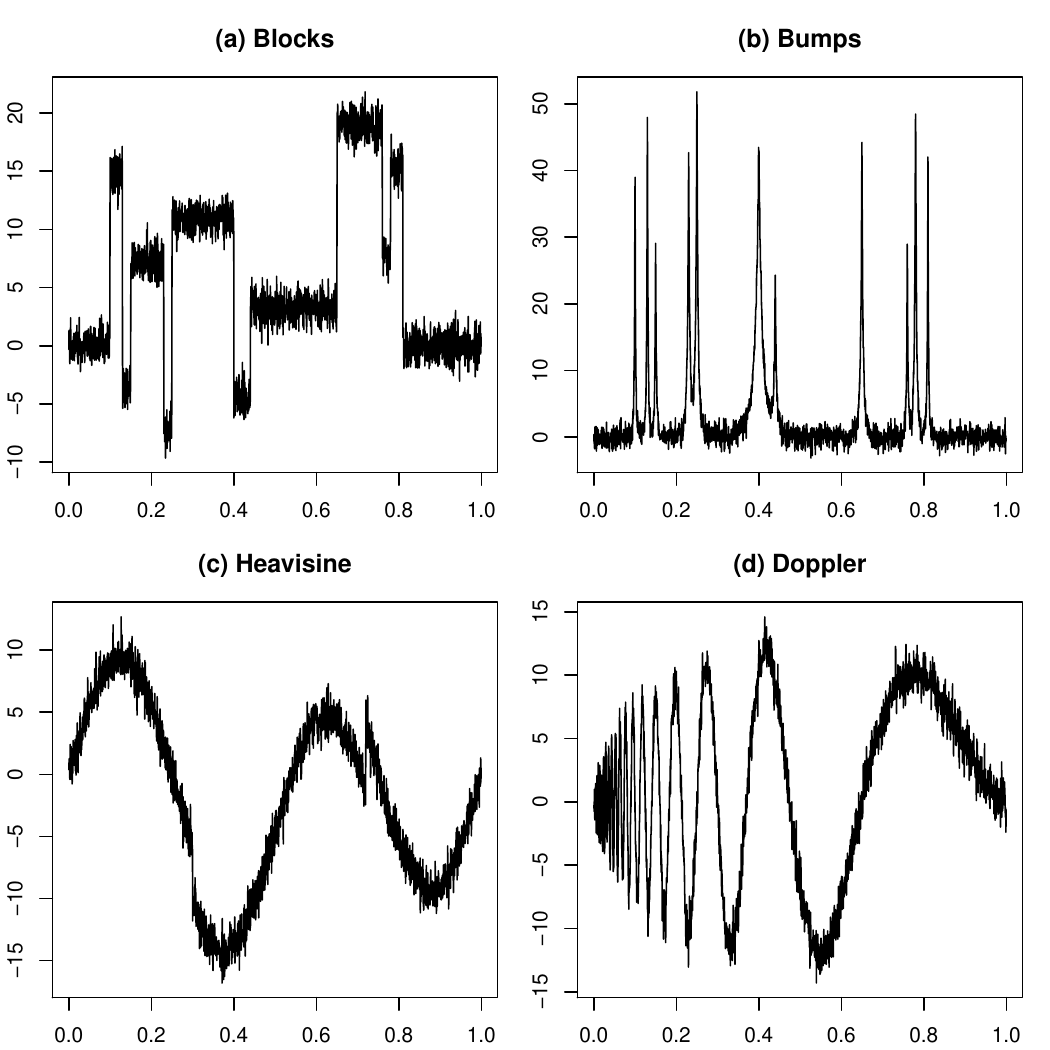}
\caption{Noisy samples from the functions in
  \autoref{fig:functions}.}
\label{fig:noise}
\end{figure}

\begin{figure}
\centering
\includegraphics[width=\textwidth,page=4]{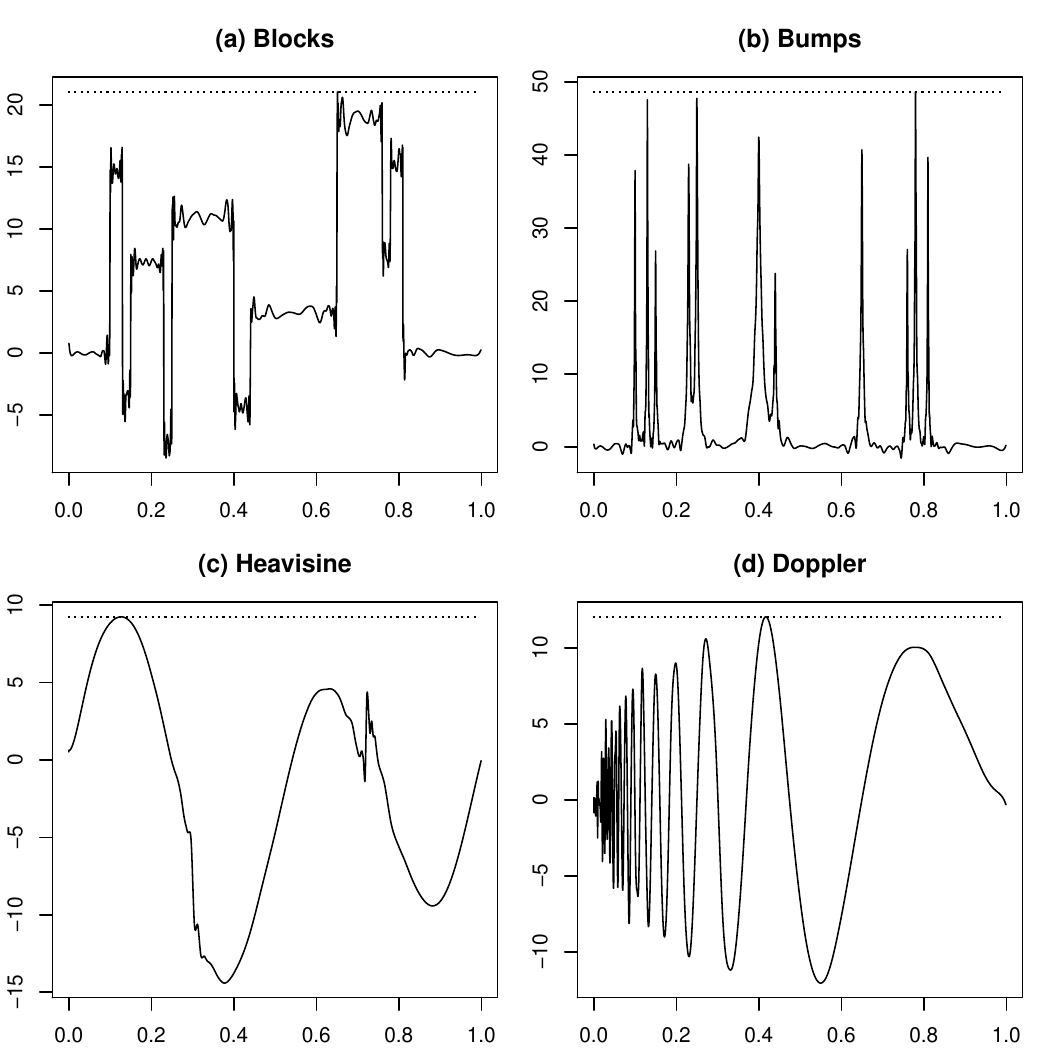}
\caption{Fixed-design estimates of the functions in
  \autoref{fig:functions}.}
\label{fig:fixed}
\end{figure}

\autoref{fig:adaptive} plots typical results of using
\autoref{alg:sas} under these conditions. The algorithm was again
given access to \(n=2^{11}\) observations, with \(\sigma = 1;\) we set
\(n_0 = 2^6,\) and chose the parameters \(\kappa,\ \lambda,\ \tau,\)
and \(\hat \sigma_n\) as in \autoref{sec:parameters}. We used the
family of wavelet bases described by \citet{cohen_wavelets_1993}, and
implemented in \citet{nason_wavethresh:_2010}; we took wavelets with
\(N = 8\) vanishing moments, set \(j_0 = 5,\) and \(\jmax=\max(j_0+1,
\lfloor n/\log(n)\rfloor).\)

\begin{figure}
\centering
\includegraphics[width=\textwidth,page=5]{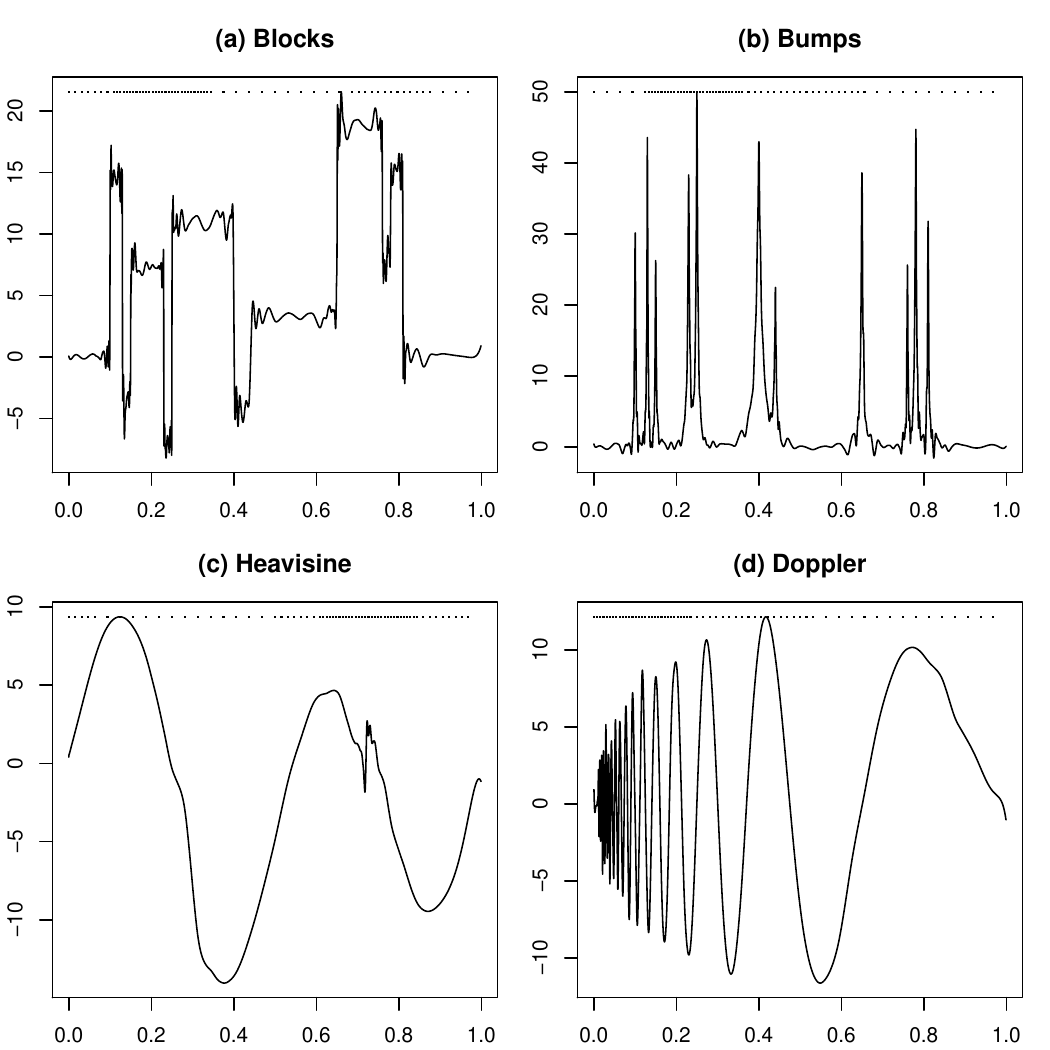}
\caption{Adaptive-sampling estimates of the functions in
  \autoref{fig:functions}.}
\label{fig:adaptive}
\end{figure}

The dots along the top of each plot are drawn proportionally to the
number of design points. We can see that, for the Heavisine and
Doppler functions, the adaptive design concentrated in the regions
where the function is rough; as a result, the adaptive estimates are
noticeably better at recovering the shape of these curves.

For the Blocks and Bumps functions, which have more complicated
patterns of spatial inhomogeneity, with these measurements the
adaptive design was not able to locate all the areas where the
functions are rough. However, we might expect performance on all the
above functions to improve as the number of design points increases;
to this end, we next considered performance with up to \(n = 2^{14}\)
design points.

At this level of detail, it becomes harder to visually compare
estimates; instead, to numerically measure the spatial adaptivity of
our estimates, we evaluated procedures in terms of their maximum error
over \([0, 1],\) approximated by
\[\max_{x \in 2^{-j}\Z \cap [0, 1)} \abs{\hat f_n(x) -
  f(x)}\] for \(j\) large. In the following, we took \(j = 17,\) to
avoid biasing the performance measure towards a uniform design.

\autoref{fig:performance} compares the performance of the two methods
on the Doppler function, with \(\sigma = 1;\) the values plotted are
sample medians after 250 runs, together with 95\% confidence intervals
for the true median. We can see that for \(n\) large, the adaptive
design significantly outperforms the uniform one, consistent with a
difference in the asymptotic rate of estimation.

\begin{figure}
\centering
\includegraphics[width=\textwidth]{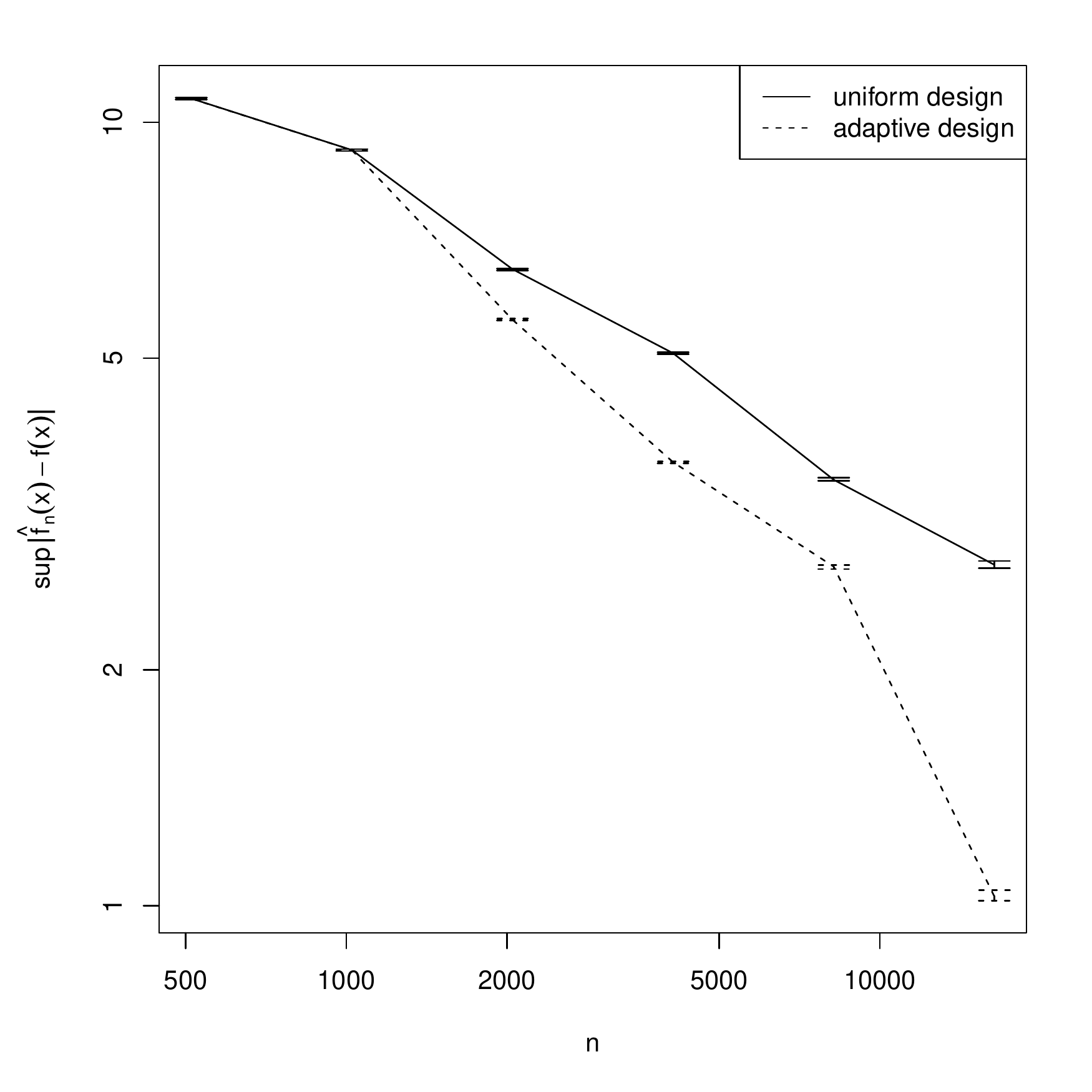}
\caption{Log-log plot of empirical performance on the Doppler function.}
\label{fig:performance}
\end{figure}

\autoref{tab:performance} compares performance on all the functions in
\autoref{fig:functions}, given \(n = 2^{14}\) observations, and
varying levels of \(\sigma.\) We again report sample medians after 250
runs, together with the \(p\)-value of a two-sided Mann-Whitney-U test
for difference in medians. (We note that the large errors reported for
the Blocks function are due to the large discontinuities present,
which are difficult to estimate uniformly over \([0, 1].\))

\begin{table}
\centering
\begin{tabular}{rrrr}
\toprule
          & uniform design & adaptive design & \(p\)-value \\
\midrule
\(\sigma = 0.5\) &&& \\
\midrule
Blocks    & {\bf 13.284} & 13.3882 & \(<0.001\)\\
Bumps     &  3.553 &  {\bf 3.0086} & \(<0.001\)\\
Heavisine &  2.646 &  {\bf 2.5902} & \(<0.001\)\\
Doppler   &  1.783 &  {\bf 0.5138} & \(<0.001\)\\
\midrule
\(\sigma = 1\) &&& \\
\midrule
Blocks    & {\bf 12.355} &  12.799 & \(<0.001\)\\
Bumps     &  5.721 &   {\bf 5.487} & \(<0.001\)\\
Heavisine &  3.260 &   {\bf 3.204} & \(<0.001\)\\
Doppler   &  2.725 &   {\bf 1.028} & \(<0.001\)\\
\midrule
\(\sigma=2\) &&& \\
\midrule
Blocks    & 11.121 &  10.988 & 0.428\\
Bumps     &  8.947 &   {\bf 7.964} & \(<0.001\)\\
Heavisine &  3.053 &   3.061 & 0.815\\
Doppler   &  3.621 &   {\bf 2.984} & \(<0.001\)\\
\bottomrule
\end{tabular}
\caption{Empirical performance on the functions in
  \autoref{fig:functions} for \(n=2^{14}.\)}
\label{tab:performance}
\end{table}

We can see that for the Blocks function, the uniform design fared
slightly better, as the adaptive algorithm still struggled to choose a
good design. However, for the other three functions, adaptive sampling
provided a significant improvement; the improvement was largest for
small \(\sigma,\) but still significant for two of the three functions
even with large \(\sigma.\) We thus conclude that adaptive sensing can
be of value in nonparametric regression whenever the function \(f\)
may be spatially inhomogeneous.

\acks{ We would like to thank Richard Nickl and several anonymous
  referees for their valuable comments and suggestions.}


%% file: sasnr-proofs.tex
\appendix

\section{Proofs}
\label{sec:proofs}

We now provide proofs of our results. We consider separately the
results describing our detectability condition; the negative results,
which establish lower bounds; and the constructive results, which
control the performance of \autoref{alg:sas}.

\subsection{Results on detectability}
\label{sec:det-proofs}

We now prove our first result, that detectability is a generic
property of locally H\"older classes.

\begin{proof}[Proof of \autoref{prop:negligible-set}]
  We will show any ball \(B_\varepsilon(f) \subseteq C^s(M, I)\)
  contains a sub-ball \(B_{\varepsilon/3}(f')\) disjoint with
  \(\mathcal D.\) Given \(\varepsilon \in (0, M),\) and \(f \in C^s(M,
  I)\) having wavelet coefficients \(\alpha_{j_0,k}\) and \(\beta_{j,
    k},\) define a function \(f'\) with wavelet coefficients
  \(\alpha_{j_0,k}\) and \(\beta_{j,k}'\), where
    \[\beta_{j,k}' \coloneqq g\left(\beta_{j,k},\, \tfrac23\varepsilon 
      2^{-j\left(s+\frac12\right)}\right),\] for
    \[g(\beta, x) \coloneqq
    \begin{cases}
      x, & \beta \in [0, x],\\
      -x, & \beta \in [-x, 0),\\
      \beta, & \text{otherwise.}
    \end{cases}\]%
    We then have \(\norm{f-f'}_{C^s(I)} \le \tfrac23 \varepsilon,\) so
    \( B_{\varepsilon/3}(f') \subset B_\varepsilon(f).\) Furthermore,
    every \(\abs{\beta_{j,k}'} \ge \tfrac23 \varepsilon
    2^{-j\left(s+\frac12\right)},\) so for any function \(f'' \in
    B_{\varepsilon/3}(f'),\) having wavelet coefficients
    \(\alpha_{j_0,k}''\) and \(\beta_{j,k}'',\) every
    \(\abs{\beta_{j,k}''} \ge \tfrac13 \varepsilon
    2^{-j\left(s+\frac12\right)}.\) Hence, for \(t = \varepsilon/3M,\)
    \begin{multline*}
      \forall\ j \in \N,\ j \ge \lceil \tfrac{j_0}t \rceil,\ k : S_{j,k} \cap I \ne \emptyset,\ j' =
      \lfloor tj \rfloor,\ k',\\
      \abs{\beta_{j',k'}''} \ge t M 2^{-j'\left(s+\frac12\right)} \ge
      (j'/j) 2^{(j-j')\left(s+\frac12\right)}\abs{\beta_{j, k}''},
    \end{multline*}%
    giving \(f'' \in D^s_t(M, I).\)
\end{proof}

\subsection{Negative results}
\label{sec:neg-proofs}

We begin our negative results by showing that, under a uniform design,
restricting to sparse detectable functions does not alter the minimax
rate of estimation. We will require \hyperref[lem:fano]{Fano's lemma},
which relates the probability of misclassifying a signal to the
Kullback-Leibler divergence between the alternatives
\citep[\S2.7.1]{tsybakov_introduction_2009}. Given probability
measures \(\P\) and \(\mathbb Q\) on \(\R^d,\) having densities \(p\)
and \(q\) respectively, define the Kullback-Leibler divergence from
\(\P\) to \(\mathbb Q,\)
\[D(\P\mid\mid \mathbb Q) \coloneqq \int p(x) \log\left(\frac{p(x)}{q(x)}\right)
dx.\]
\begin{lemma}[Fano's lemma]
  \label{lem:fano}
  Let \(X \in \R^d\) have distribution \(\P_i,\) for some \(i \in \N,\
  1 \le i \le k,\) and let \(\psi(X)\) be an estimate of \(i.\) Then
  \[\sup \nolimits_i \P_i(\psi(X) \ne i) \ge 1 -
  \frac{\beta + \log(2)}{\log(k-1)},\] where \[\beta \coloneqq
  \frac1{k^2}\sum_{i,j=1}^k D(\P_i \mid \mid \P_j).\]
\end{lemma}

We also make the definition \[\norm{f}_{I, \infty} \coloneqq \sup_{x
  \in I} \abs{f(x)},\] for functions \(f : [0, 1] \to \R,\) and \(I
\subseteq [0, 1].\)

\begin{proof}[Proof of \autoref{thm:fixed-lower-bound}]
  The argument proceeds as a standard minimax lower bound; we
  construct functions \(f_{n,k} \in \mathcal F\) a distance \(c_n\)
  apart, and show we can only distinguish between them when \(c_n
  \gtrsim (n/\log(n))^{-\alpha(s)}.\)

  Choose \(j \in \N\) so that
  \[2^{j} \sim (n/\log(n))^{1/(2s+1)},\] and define a sequence \(j_1,
  \dots, j_m\) by \[j_m \coloneqq j,\qquad j_{i-1} \coloneqq \lfloor t
  j_i \rfloor,\qquad j_0 \le j_1 < \lceil \tfrac{j_0}t \rceil.\] Since
  there are \(n\) design points \(x_i,\) for \(n\) large we must have
  an interval \(S_{j_{m-1},k} \subseteq I\) containing \(\lesssim
  2^{-j_{m-1}}n\) of them; we will assume without loss of generality
  that this interval is always \(S_{j_{m-1},0}.\) We then consider
  functions \(f_{n,k} \in \mathcal F,\) given by
  \[f_{n,k} \coloneqq \sum_{i=1}^{m-1}
  M2^{-j_i\left(s+\frac12\right)}\psi_{j_i,0} +
  C2^{-j_m\left(s+\frac12\right)}\psi_{j_m,k},\]%
  where \(k \in \Z,\ 0 \le k < 2^{j_m-j_{m-1}},\) and \(C \in (0,
  M]\) is a constant to be determined.

  Suppose \(\norm{\hat f_n - f}_{I,\infty} = O_p(c_n),\) uniformly
  over \(\mathcal F,\) for a sequence \(c_n \not \gtrsim
  (n/\log(n))^{-\alpha(s)}.\) Then on a subsequence, we have \(c_n =
  o\left((n/\log(n))^{-\alpha(s)}\right);\) passing to the subsequence, we
  may assume this is true for all \(n.\) Since the \(\psi_{j_m,k}\)
  have distinct supports, for \(n\) large we have
  \[\min_{k \ne k'} \norm{f_{n,k} - f_{n,k'}}_{I,\infty} \gtrsim
  C2^{-j_ms} \gtrsim C(n/\log(n))^{-\alpha(s)}.\] Thus for \(n\) large,
  \(\hat f_n\) can distinguish between the \(f_{n,k}\) with
  arbitrarily high probability. Let \(\P_k\) denote the distribution
  of the observations when \(f = f_{n,k};\) then any estimate
  \[\hat k_n \coloneqq \underset{k}{\arg \min} \norm{\hat f_n -
    f_{n,k}}_\infty\]
  of \(k\) satisfies
  \[\sup_k \P_k(\hat k_n \ne k) \to 0\]
  as \(n \to \infty.\)
 
  However, for \(k,\, k' \in \Z,\ 0 \le k, k' < 2^{j_m-j_{m-1}},\) the
  Kullback-Leibler divergence from \(\P_k\) to \(\P_{k'}\) is
  \begin{align*}
    D(\P_{k} \mid\mid \P_{k'}) &=
    \sum_{i=1}^n \frac1{2\sigma^2} \left(f_{n,k}(x_i) - f_{n,k'}(x_i)\right)^2 \\
    &=  \frac1{2\sigma^2} C^2 2^{-j_m(2s+1)} \sum_{i=1}^n \left(\psi_{j_m,k}(x_i) -
      \psi_{j_m,k'}(x_i)\right)^2,\\
    &\lesssim
    C^22^{-j_m(2s+1)}\sum_{i=1}^n2^{j_m}\left(1(x_i
      \in S_{j_m,k}) + 1(x_i \in S_{j_m,k'})\right).
  \end{align*}
  Thus as \(\lesssim 2^{-j_{m-1}}n\) design points lie within
  \(\bigcup_k S_{j_m,k} \subseteq S_{j_{m-1},0},\)
  \begin{align*}
    2^{-2(j_m-j_{m-1})}\sum_{k,k'}D(\P_{k} \mid\mid \P_{k'})
    &\lesssim
    C^22^{j_{m-1}-j_m(2s+1)}\sum_{i=1}^n\sum_k1(x_i
      \in S_{j_m,k})\\
    &\lesssim C^22^{-j_m(2s+1)}n \lesssim
    C^2\log(n).
  \end{align*}
  As there are \(2^{j_m-j_{m-1}}\) alternatives for \(k,\) and \(j_m-j_{m-1}
  \lesssim \log(n),\) when \(C\) is small this contradicts
  \hyperref[lem:fano]{Fano's lemma}.
\end{proof}

We next provide similar lower bounds for adaptive-sensing algorithms.
In this case, the argument from \hyperref[lem:fano]{Fano's lemma}
presents difficulties; instead, we will argue using
\hyperref[lem:assouad]{Assouad's lemma}, which bounds the accuracy of
estimation over a cube in terms of Kullback-Leibler divergences
\citep[\S2.7.2]{tsybakov_introduction_2009}. While this choice
leads to the loss of a log factor in the results proved, it allows us
to give bounds which apply also for adaptive sensing.

\begin{lemma}[Assouad's lemma]
  \label{lem:assouad}
  Let \(\Omega \coloneqq \{0, 1\}^m,\) and for \(p \in (0, \tfrac12],\)
  define a distribution \(\pi\) over \(\Omega,\)
  \[\pi(\omega) = p^{\sum_i \omega_i}(1-p)^{\sum_i (1-\omega_i)}.\]
  For each \(\omega \in \{0, 1\}^m,\) let \(\P_\omega\) be a
  probability measure on \(\R^d,\) and \(\E_\omega\) the corresponding
  expectation. Then for any estimator \(\hat \omega\) of \(\omega,\)
  \[\sum_{\omega \in \Omega} \pi(\omega) \E_\omega \rho(\hat \omega, \omega) \ge
  pm\left(1 - \sqrt{ \frac1{2m}\sum_{i=1}^{m}
      \sum_{\omega \in \Omega}\pi(\omega)D(\P_\omega \mid\mid
      \P_{\omega^{i}})}\right),\]
  where \(\rho(\omega, \omega')\) is the Hamming distance, and
  \(\omega^{i}\) equals \(\omega\) except in the \(i\)th coordinate.
\end{lemma}

Our argument then proceeds as in
\citet{arias-castro_fundamental_2011}; we will start with a simple
lemma on the truncated expectation of binomial random variables.

\begin{lemma}
  \label{lem:bin-exp}
  If \(X \sim \Bin(n, p),\) then \(\E X 1(X > 2np) = o(np)\) as \(np
  \to \infty.\)
\end{lemma}

\begin{proof}
  Considering the mass function of \(X,\) we have
  \[\E X 1(X > 2np) = np \P (Y > 2np - 1),\]
  where \(Y \sim \Bin(n-1, p).\) From Cheybshev's inequality, we then obtain
  \[\P(Y > 2np - 1) \le \frac1{(n-1)p} \to 0. \qedhere\]
\end{proof}

We next give a lemma which allows us to control the performance of
adaptive-sensing algorithms.

\begin{lemma}
  \label{lem:adaptive-bounds}
  Given sequences \(j_n, k_n, l_n \in \N,\) with \(2^{j_n} \ge k_n \ge
  l_n \to \infty,\) \(l_n = o(k_n),\) let \(K_n \coloneqq \{0, \dots,
  k_n-1\},\) and \(\mathcal K_n \coloneqq \left\{ K \subseteq K_n :
    \abs{K} \le l_n\right\}.\) Given also functions \(f_n \in
  L^2([0,1]),\) and a sequence \(\mu_n\) satisfying \(0 \le \mu_n \le
  C\sqrt{k_n/n}\) for \(C\) small, define
  \[f_{n,K} \coloneqq f_n + \sum_{k \in K}\mu_n2^{-\frac12
    j_n}\psi_{j_n,k},\qquad K \subseteq K_n.\] Finally, let \(I
  \subseteq [0, 1]\) be an interval satisfying \(\cup_{k=0}^{k_n-1}
  S_{j_n,k} \subseteq I,\) for \(n\) large.

  Suppose that an adaptive-sensing algorithm with estimator \(\hat
  f_n\) satisfies
  \[\norm{\hat f_n - f_{n,K}}_{I,\infty} = O_p(c_n),\]
  uniformly over \(K \in \mathcal K_n.\) Then
  \[c_n \gtrsim \mu_n.\]
\end{lemma}

\begin{proof}
  We have
  \[\underset{K \ne K'}{\inf_{K, K' \in \mathcal K_n}}
  \norm{f_{n,K}-f_{n,K'}}_{I,\infty} \gtrsim \mu_n,\] since for \(n\)
  large, \(f_{n,K} - f_{n,K'}\) must be given by a single wavelet on
  some interval \(2^{-j_m}[l,l+1) \subseteq I.\) Suppose \(\norm{\hat
    f_n - f_{n,K}}_{I,\infty} = O_p(c_n),\) uniformly over \(\mathcal
  K_n,\) for a sequence \(c_n \not \gtrsim \mu_n.\) On a subsequence,
  we have \(c_n = o(\mu_n);\) passing to that subsequence, we may
  assume this is true for all \(n.\) Any estimate
  \[\hat K_n \coloneqq \underset{K \in \mathcal K_n}{\arg\min} \norm{\hat f_n
    - f_{n,K}}_{I,\infty}\] of \(K\) then satisfies
  \[\sup_{K \in \mathcal K_n} \P_K(\hat K_n \ne K) \to 0\]
  as \(n \to \infty;\) we will show this contradicts
  \hyperref[lem:assouad]{Assouad's lemma}.

  Define a distribution \(\pi\) over \(K \subseteq K_n,\) letting the
  variables \(1(k \in K),\) \(k \in K_n,\) be i.i.d., so that
  \[\abs{K} \sim \Bin\left(k_n, \frac{l_n}{2k_n}\right).\]
  Denote by \(\E_\pi\) the expectation when we first draw \(K\)
  according to \(\pi,\) and then observe under \(\P_K.\)
  Since \(\abs{\hat K_n \triangle K} \le l_n + \abs{K}\)
  (where \(\triangle\) denotes symmetric difference), we have
  \begin{align*}
      \E_\pi\abs{\hat K_n \triangle K} &\le
    2\left(\E_\pi\abs{K}1(\abs{K} > l_n) + l_n\sup_{K \in
        \mathcal K_n} \P_K(\hat K_n \ne K)\right)\\
    &=o(l_n),
  \end{align*}
  using \autoref{lem:bin-exp}.

  However, for \(K \subseteq K_n,\) we also have
  \begin{align*}
    \sum_{k=0}^{k_n-1}\E_\pi D(\P_K \mid \mid \P_{K \triangle \{k\}})
    &= \sum_{k=0}^{k_n-1} \E_\pi \sum_{i=1}^n
    \frac{1}{2\sigma^2}\left(f_{n,K}(x) - f_{n,K \triangle \{k\}}(x)\right)^2\\
    &\lesssim \mu_n^2\sum_{i=1}^n\E_\pi \sum_{k=0}^{k_n-1}
    1(x_i \in S_{j_n,k})\\
    &\lesssim C^2k_n.
  \end{align*}
  Thus for \(C\) small, by \hyperref[lem:assouad]{Assouad's lemma},
  \[\E_\pi\abs{\hat K_n \triangle K} \gtrsim l_n,\]
  giving us a contradiction.
\end{proof}

We may now proceed to prove our adaptive-sensing lower bounds,
applying this lemma in several different contexts.

\begin{proof}[Proof of \autoref{thm:adaptive-lower-bound}] 
  We first prove \(c_n \gtrsim n^{-\alpha(r')}.\) Choose \(j \in \N\)
so that
  \[2^{j} \sim n^{1/(2r'+1)},\] and define a sequence \(j_1, \dots,
  j_m\) by
  \[j_m \coloneqq j,\qquad j_{i-1} \coloneqq \lfloor t j_i
  \rfloor,\qquad j_0 \le j_1 < \lceil \tfrac{j_0}t \rceil.\] We consider
  functions
  \[f_{n,K} \coloneqq \sum_{i=1}^{m-1}\sum_{k=0}^{2^{j_i}-1}
  M2^{-j_i\left(r+\frac12\right)}\psi_{j_i,k} + \sum_{k\in
    K}C2^{-j_m\left(r' + \frac12\right)}\psi_{j_m,k},\]%
  where \(K \subseteq K_n \coloneqq \{k \in \Z : 0 \le k < 2^{j_m},\ 
  S_{j_m,k} \subseteq I\},\) and \(C \in (0, M]\) is small. Let
  \[\mathcal K_n \coloneqq \left\{K \subseteq K_n : \abs{K} \le 2^{tj_m}\right\},\] so if \(K
  \in \mathcal K_n,\) then \(f_{n,K} \in \mathcal F.\) Applying
  \autoref{lem:adaptive-bounds}, we obtain
  \[c_n \gtrsim n^{-\alpha(r')}.\]

  To show \(c_n \gtrsim n^{-\alpha(s')},\) we make a similar argument,
  this time setting \[2^j \sim n^{1/(1-ptu)(2s'+1)},\] and defining
  \(j_1, \dots, j_m\) as before.  For \(n\) large enough, we must have
  an interval \(S_{j_m, k} \subseteq I;\) we will assume without loss of
  generality this interval is \(S_{j_m,0}.\) We then consider
  functions
  \[f_{n,K} \coloneqq \sum_{i=1}^{m-1}\sum_{k=0}^{2^{j_i(1-pu)}-1}
  M2^{-j_i\left(s+\frac12\right)}\psi_{j_i,k} + \sum_{k\in
    K}M\varepsilon_n2^{-j_m\left(s+\frac12\right)}\psi_{j_m,k},\]%
  for \(K \subseteq K_n \coloneqq \left\{k \in \Z : 0 \le k < 2^{j_m -
      j_{m-1}pu}\right\},\) and \(\varepsilon_n \in (0, 1]\) a
  sequence to be determined, with \(\varepsilon_n \to 0.\) Let
  \[\mathcal K_n \coloneqq \left\{K \subseteq K_n : \abs{K} \le
    \varepsilon_n^{-p}\right\},\] so again if \(K \in \mathcal K_n,\)
  \(f_{n,K} \in \mathcal F.\) For any \(\varepsilon_n\) decreasing
  slowly enough, we may apply \autoref{lem:adaptive-bounds}, obtaining
  \(c_n \gtrsim \varepsilon_n n^{-\alpha(s')};\) we must therefore
  have
  \[c_n \gtrsim n^{-\alpha(s')}.\qedhere\]
\end{proof}

\begin{proof}[Proof of \autoref{thm:holder-lower-bound}]
  This follows as a corollary of \autoref{thm:adaptive-lower-bound}.
  We apply the theorem to classes
  \[\mathcal F = B^s_{\infty, \infty}(M) \cap D^s_t(J),\]
  noting that \(\mathcal F \subseteq C^s(M) \subseteq C^s(M, I).\)
\end{proof}

\begin{proof}[Proof of \autoref{thm:besov-lower-bound}]
  This follows similarly to the second half of
  \autoref{thm:adaptive-lower-bound}. If we instead set \(2^j \sim
  n^{1/(2s+1)},\)
  \[f_{n,K} \coloneqq \sum_{k\in
    K}M\varepsilon_n2^{-j\left(s+\frac12\right)}\psi_{j,k},\] and
  \(K_n \coloneqq \{k \in \Z : 0 \le k < 2^j,\ S_{j,k}
  \subseteq I\},\) by the same argument we have \(c_n \gtrsim
  n^{-\alpha(s)}.\)
\end{proof}

\subsection{Constructive results}
\label{sec:con-proofs}

We now prove that \autoref{alg:sas} attains near-optimal rates of
convergence. Our proofs involve a series of lemmas; the first shows
that the algorithm chooses design points so that the discrepancy from
the target density \(p_m,\) to the effective density \(q_m,\) remains
bounded.

\begin{lemma}
  \label{lem:discrepancy-bounded}
  Let the design points \(x_n\) be chosen by \autoref{alg:sas}, and
  suppose \[1+2C \le \frac{n_{m}}{n_{m-1}} \le D,\] for constants \(C,
  D > 0.\) Then for \(m\) larger than a fixed constant,
  \[q_{l,m} \ge \frac{C}{D}p_{l,m} \ge \frac{C}{D}A\lambda.\]
\end{lemma}

\begin{proof}
  Suppose at stage \(m,\) the new design \(\xi_m\) included dyadic
  grids \(2^{-i}\Z \cap I_{l,m},\) where the indices \(i\) were chosen
  to ensure that there were at least
  \[r_{l,m} \coloneqq Cn_{m-1}2^{-\jmaxm}p_{l,m}\] such points in the
  interval \(I_{l,m},\) \(2^{i-\jmaxm} \ge r_{l,m}.\) We will show
  that for this design, the discrepancy from \(p_m\) to \(q_m\) is
  bounded, and that \autoref{alg:sas} must do at least this well,

  Since \(p_{l,m} \ge A\lambda,\) for \(m\) large we have \(r_{l,m}
  \ge 1.\) This design would thus include the points \(2^{-\jmaxm}\Z
  \cap [0, 1),\) and would require at most
  \[\sum_{l=0}^{2^\jmaxm-1} 2r_{l,m} \le 2Cn_{m-1} \le n_m-n_{m-1}\]
  additional design points; we would then have
  \[q_{l,m} \ge \frac{2^{\jmaxm}r_{l,m}}{n_m} \ge \frac{C}{D}
  p_{l,m}.\] Since \autoref{alg:sas} includes the points
  \(2^{-\jmaxm}\Z \cap [0, 1),\) and then chooses its remaining design
  points to minimise \[\max_{l=0}^{2^\jmaxm-1} p_{l,m}/q_{l,m},\] the
  same must be true for its choice of design points.  The final
  inequality follows as \(p_{l,m} \ge A\lambda.\)
\end{proof}

We next consider the operation of \autoref{alg:sas} under a
deterministic noise model. Let \(e_n\) be given by \eqref{eq:e-defn},
and suppose that our estimates \(\hat \alpha_{j_0,k},\) \(\hat
\beta_{j,k}\) are instead chosen adversarially, subject to the
conditions that
\begin{equation}
\label{eq:coeff-bounds}
\abs{\hat \alpha_{j_0,k} - \alpha_{j_0,k}} \le e_n(j_0,k), \qquad 
\abs{\hat \beta_{j,k} - \beta_{j,k}} \le e_n(j,k),
\end{equation}
for all \(k,\) and \(j_0 \le j < \jmax.\) We then show
that, in this model, the target densities \(p_m\) will be large in
regions where the wavelet coefficients \(\beta_{j,k}\) are large.

\begin{lemma}
  \label{lem:densities-good}
  In the deterministic noise model, let \(\mathcal F\) be given by
  \eqref{eq:f-defn}, for \(p < \infty.\) For any \(f \in \mathcal F,\)
  \(m, j \in \N,\) \(j_0 \le j < \jmaxmm,\) and \(k \in \Z,\ 0 \le k <
  2^j,\) suppose
  \[\abs{\beta_{j,k}} \ge (\kappa + 1)e_{n_{m-1}}(j,k).\]
  Then for \(I_{l,m} \subseteq S_{j,k},\) we have
  \[p_{l,m} \gtrsim \frac{2^{jp\left(r+\frac12\right)}\abs{\beta_{j,k}}^p}{\log(n_m)^2},\]
  uniformly in \(f,\,m,\,j,\) and \(k.\)
\end{lemma}

\begin{proof}
  We first establish that, for non-thresholded coefficients, our
  estimates \(\hat \beta_{j,k}^T\) are of comparable size to the
  \(\beta_{j,k}.\) For \(n \coloneqq n_{m-1},\) we have
  \[\abs{\hat \beta_{j,k}} \ge \abs{\beta_{j,k}} - e_n(j,k) \ge \kappa e_n(j,k),\]%
  so \(\hat \beta_{j,k}^T \ne 0.\)  
  Suppose \(\abs{\hat \beta_{j,k'}^T} \ge \abs{\hat \beta_{j,k}^T}\)
  for some \(k'.\) Then \(\hat \beta_{j,k'}^T \ne 0,\)
  so
  \[\abs{\beta_{j,k'}} \ge \abs{\hat
      \beta_{j,k'}} - e_n(j,k') \ge (\kappa - 1)e_n(j,k'),\]%
  and
  \[\abs{\hat \beta_{j,k'}} \le \abs{\beta_{j,k'}} + e_n(j,k') \lesssim \abs{\beta_{j,k'}}.\]
  We thus obtain that
  \begin{equation}
    \label{eq:beta-large}
    \abs{\beta_{j,k'}} \gtrsim \abs{\hat \beta_{j,k'}}
    \ge \abs{\hat \beta_{j,k}}
    \ge \abs{\beta_{j,k}} - e_n(j,k)
    \gtrsim \abs{\beta_{j,k}}.
  \end{equation}

  We may then conclude that, for such coefficients, the noise has
  little effect on the target density. Since \(f \in
  B^r_{p,\infty}(M),\) the number of coefficients \(\beta_{j,k'}\)
  satisfying \eqref{eq:beta-large} can be at most
  \[2^{-jp\left(s+\frac12-\frac1p\right)}\abs{\beta_{j,k}}^{-p},\] up
  to constants. Thus, for any \(I_{l,m} \subseteq S_{j,k},\) the
  target density
  \[p_{l,m} \gtrsim \frac{2^{jp\left(s+\frac12\right)}
    \abs{\beta_{j,k}}^p}{(\jmax)^2}.\]%
  As \(\jmax \lesssim \log(n_m),\) and these bounds are uniform over
  \(f,\,m,\,j,\) and \(k,\) the result follows.
\end{proof}

Next, we prove a technical lemma, which shows that each term in the
wavelet series of \(f\) will lie within the support of larger terms at
lower resolution levels. For a given function \(f,\) if \(j, k, j',
k'\) satisfy \eqref{eq:parent-cond}, we will say that
\(\beta_{j',k'}\) is a {\em parent} of \(\beta_{j,k}.\)

\begin{lemma}
  \label{lem:parent-wavelets}
  Let \(\mathcal F\) be given by \eqref{eq:f-defn}, and pick \(\jmin
  \in \N,\)
  \[2^{\jmin} \sim \left(n/\log(n)^{1-\frac4p}\right)^{1/\left(2tr +
      2(1-t)s+1\right)}.\] Then for any \(f \in \mathcal F,\) \(n, j \in
  \N,\) \(j_0 \le j < \jmax,\) and \(k : S_{j,k} \cap I \ne
  \emptyset,\) there is a sequence
  \(\beta_{j_1,k_1},\dots,\beta_{j_d,k_d}\) of wavelet coefficients of
  \(f,\) satisfying:
  \begin{enumerate}
  \item \(\beta_{j_i,k_i}\) is a parent of
    \(\beta_{j_{i+1},k_{i+1}};\)
  \item \(j_1 \le \jmin,\) \(j_d = j,\) \(k_d = k;\) and
  \item \(d\) is bounded by a fixed constant.
  \end{enumerate}
\end{lemma}

\begin{proof}
  If \(j \le \jmin,\) we are done. If not, since \(f \in \mathcal F,\)
  we must have a coefficient \(\beta_{j',k'}\) which is a parent of
  \(\beta_{j,k}.\) Choose \(\beta_{j', k'}\) so that \(j'\) is
  minimal, and if also \(j' > \jmin,\) let \(\beta_{j'',k''}\) be a
  parent of \(\beta_{j',k'}.\) We will show that we may continue in
  this fashion until we choose a coefficient \(\beta_{j_1,k_1}\) with
  \(j_1 \le \jmin.\)

  If \(j' > \jmin,\) we have that \(j'' < j,\) \(S_{j'',k''} \supset
  S_{j,k},\) and
  \[\abs{\beta_{j'',k''}} \ge
  (j''/j)2^{(j-j'')\left(s+\frac12\right)}\abs{\beta_{j, k}}.\] If also \(j''
  \ge tj,\) this would make \(\beta_{j'',k''}\) a parent of
  \(\beta_{j,k},\) contradicting our choice of \(j'.\) Thus \(j'' <
  tj.\) Since every two steps, we reduce \(j\) by a factor of \(t,\)
  and \(\jmax/\jmin\) tends to a constant, it takes at most a constant
  number of steps to reach \(\jmin.\)
\end{proof}

We may now show that, in this model, the algorithm will ensure all large
coefficients are estimated accurately.
\begin{lemma}
  \label{lem:coeffs-accurate}
  In the deterministic noise model, let the design points \(x_n\) be
  chosen by \autoref{alg:sas}, and let \(\mathcal F\) be given by
  \eqref{eq:f-defn}. For \(n = n_m\) and \(C > 0\) large, not
  depending on \(f,\) the following results hold for all \(j,k \in
  \Z,\ \jmin \le j < \jmax,\ 0 \le k < 2^j.\)
  \begin{enumerate}
  \item \(e_n(j,k) \le C(n/\log(n))^{-\frac12}.\)
  \item If \(f \in \mathcal F,\) \(p < \infty,\) and \(S_{j,k} \cap I
    \ne \emptyset,\) define
    \[v \coloneqq \frac{tr+(1-t)s + \frac12}{1+\frac2p},\] and
    \[\delta_n \coloneqq C\left(n/\log(n)^3\right)^{-1/(p+2)}.\]
    Then 
    \[\abs{\beta_{j,k}} \ge (\kappa + 1)2^{-jv}\delta_n \quad \implies \quad e_n(j,k) \le
    2^{-jv}\delta_n.\]
  \end{enumerate}
\end{lemma}

\begin{proof}
  We consider the two parts separately.  For part (i), for \(n\) large
  we may use the fact that the effective density is bounded below; we
  have that by \autoref{lem:discrepancy-bounded}, \(q_{l,m} \gtrsim
  1,\) so \(2^{i_n(j, k)} \gtrsim n,\) and \[e_n(j,k) \lesssim (n/\log
  (n))^{-\frac12}.\] The result thus holds for \(C\) large.

  For part (ii), we will argue that for \(f \in \mathcal F,\) large
  coefficients \(\beta_{j,k}\) must have large parents, which we can
  detect over the noise. We will thus place more design points in
  their support, and so estimate the \(\beta_{j,k}\) more accurately.

  Suppose \(\abs{\beta_{j,k}} \ge (\kappa + 1)2^{-jt}\delta_n,\) and
  apply \autoref{lem:parent-wavelets} to \(\beta_{j,k}.\) We then
  obtain wavelet coefficients
  \(\beta_{j_1,k_1},\dots,\beta_{j_d,k_d}\) satisfying the conditions
  of the lemma, which we choose so that \(d\) is minimal; we proceed
  by induction on \(d.\) If \(d = 1,\) then \(j \le \jmin,\) and
  \[2^{-jv}\delta_n \ge 2^{-\jmin v}\delta_n \gtrsim (n/\log
  (n))^{-\frac12}.\]%
  For \(C\) large, the claim then follows from part (i).

  Inductively, suppose the claim holds for \(d-1;\) by minimality of
  \(d,\) we may assume \(j > \jmin.\) If \(\abs{\beta_{j,k}} \ge
  (\kappa + 1)2^{-jv}\delta_n,\) then for \(n\) large we have
  \[\abs{\beta_{j_{d-1},k_{d-1}}} \gtrsim (j_{d-1} / j)
  2^{(j-j_{d-1})\left(s+\frac12-v\right)}2^{-j_{d-1}v}\delta_n \gtrsim
  2^{-j_{d-1}v}\delta_{n_{m-1}},\]%
  since \(\delta_n \gtrsim \delta_{n_{m-1}},\) \(j_{d-1}/j \ge t -
  1/\jmin \gtrsim 1,\) and
  \[s + \tfrac12 - v \ge
  \frac{s+\frac12-\frac{t}2}{1+\frac{p}2} \ge 0.\]%
  For \(n\) and \(C\) large, we may thus apply the inductive
  hypothesis at time \(n_{m-1},\)
  obtaining \[\abs{\beta_{j_{d-1},k_{d-1}}} \ge (\kappa +
  1)e_{n_{m-1}}(j_{d-1},k_{d-1}).\]%

  We then apply \autoref{lem:densities-good} to
  \(\beta_{j_{d-1},k_{d-1}},\) obtaining that, for any \(I_{l,m}
  \subseteq S_{j_{d-1},k_{d-1}},\)
  \begin{align*}
    p_{l,m} &\gtrsim \log(n)^{-2}2^{j_{d-1}p\left(r+\frac12\right)}\abs{\beta_{j_{d-1},k_{d-1}}}^p\\
    &\gtrsim \log(n)^{-2}2^{jv(p+2)}\abs{\beta_{j,k}}^p\\
    &\gtrsim \log(n)^{-2}2^{2jv}\delta_n^p.
  \end{align*}
  For \(n\) large, by \autoref{lem:discrepancy-bounded} we will have
  \[q_{l,m} \gtrsim \log(n)^{-2}2^{2jv}\delta_n^p\] for such \(l.\)
  Since \(S_{j,k} \subseteq S_{j_{d-1},k_{d-1}},\) we also have
  \(2^{i_n(j,k)} \gtrsim n\log(n)^{-2}2^{2jv}\delta_n^p,\) and
  \begin{align*}
    e_n(j,k) &\lesssim n^{-\frac12}\log(n)^{\frac32}2^{-jv}\delta_n^{-\frac12p}\\
    &\lesssim 2^{-jv}\delta_n.
  \end{align*}
  Thus for \(C\) large, the claim is also proved for \(d.\)

  From \autoref{lem:parent-wavelets}, we know the number of induction
  steps is bounded by a fixed constant, so there must be a single
  choice of \(n\) and \(C\) large enough to satisfy all the above
  requirements. As this choice is also uniform over \(f \in \mathcal
  F,\) the result follows.
\end{proof}

Using this result, we conclude that \autoref{alg:sas} attains good
rates of convergence over spatially-inhomogeneous functions.

\begin{lemma}
  \label{lem:algorithm-accurate}
  In the deterministic noise model, let the design points \(x_n\) be
  chosen by \autoref{alg:sas}, \(\mathcal F\) be given by
  \eqref{eq:f-defn}, and \(c_n\) by \eqref{eq:alg-rates}. Then
  \[\sup_{f \in \mathcal F} \norm{\hat f_n - f}_{I,\infty} =
  O\left(c_n\right).\]%
\end{lemma}

\begin{proof}
  We may bound the error in \(\hat f_n\) over \(I\) by
  \begin{multline*}
    \norm{\hat f_n - f}_{I,\infty} \lesssim \max_{S_{j_0,K} \cap I \ne
      \emptyset} \abs{\hat \alpha_{j_0,k} - \alpha_{j_0,k}}\\ +
    \sum_{j=j_0}^{\jmax-1} \max_{S_{j,k} \cap I \ne \emptyset}
    2^{\frac12j} \abs{\hat \beta_{j,k}^T - \beta_{j,k}} +
    \sum_{j=\jmax}^\infty \max_{S_{j,k} \cap I \ne \emptyset}
    2^{\frac12j} \abs{\beta_{j,k}}.
  \end{multline*}
  In the deterministic noise model, \(\abs{\hat \alpha_{j_0,k} -
    \alpha_{j_0,k}} \le e_n(j,k),\) and the thresholded coefficients
  \(\hat \beta_{j,k}^T\) fall into one of two cases.
  \begin{enumerate}
  \item If \(\hat \beta_{j,k}^T = 0,\) then
    \[\abs{\beta_{j,k}} \le \abs{\hat
      \beta_{j,k}} + e_n(j,k) \le (\kappa + 1)
    e_n(j,k),\] so \[\abs{\hat \beta_{j,k}^T - \beta_{j,k}} =
    \abs{\beta_{j,k}} \lesssim e_n(j,k).\]
  \item If \(\hat \beta_{j,k}^T \ne 0,\) then
    \[\abs{\beta_{j,k}} \ge \abs{\hat \beta_{j,k}} -
     e_n(j,k) \ge (\kappa - 1) e_n(j,k),\]
    so \[\abs{\hat \beta_{j,k}^T - \beta_{j,k}} = \abs{\hat
      \beta_{j,k} - \beta_{j,k}} \le e_n(j,k) \lesssim
    \abs{\beta_{j,k}}.\]
  \end{enumerate}
  Thus, in either case, we have
  \[\abs{\hat \beta_{j,k}^T - \beta_{j,k}} \lesssim
  \min(e_n(j,k), \abs{\beta_{j,k}}).\]

  For \(n\) large, we may then bound the error in \(\hat f_n\) using
  \autoref{lem:coeffs-accurate}; since the ratios \(n_m/n_{m-1}\) are
  bounded, it suffices to consider times \(n=n_m.\) The contribution
  from the \(\alpha_{j_0,k}\) is of order \((n/\log(n))^{-\frac12},\)
  so may be neglected.  Considering the \(\beta_{j,k}\) terms, if \(r
  \le s,\) from \autoref{lem:coeffs-accurate} and the definition of
  our detectability condition, we obtain the bounds
  \[e_n(j, k) \lesssim (n/\log(n))^{-\frac12},\qquad \abs{\beta_{j,k}}
  \lesssim 2^{-j\left(s+\frac12\right)}.\] Pick \(j_n \in \N\) so that
  \[2^{j_n} \sim (n/\log(n))^{1/(2s+1)},\] and we
  obtain
  \begin{align}
    \notag \norm{\hat f_n - f}_{I,\infty} &\lesssim \sum_{j=j_0}^{j_n-1}
    2^{\frac12j}(n/\log(n))^{-\frac12}
    + \sum_{j=j_n}^\infty 2^{-js}\\
    \label{eq:holder-error-bound}
    &\lesssim (n/\log(n))^{-\alpha(s)}.
  \end{align}

  Assume instead we have \(r > s,\) so \(p < \infty.\) From
  \autoref{lem:coeffs-accurate}, we then obtain the additional bound
  \[\min(e_n(j,k),\abs{\beta_{j,k}}) \lesssim 2^{-jv}\delta_n.\]
  If \(r' < s',\) then \(v > \frac12,\) so
  \begin{align*}
    \norm{\hat f_n - f}_{I,\infty} &\lesssim \sum_{j=j_0}^{\jmin-1}
    2^{\frac12j}(n/\log(n))^{-\frac12} + \sum_{j=\jmin}^{\jmax-1} 2^{-j\left(v-\frac12\right)}\delta_n
    + \sum_{j=\jmax}^\infty 2^{-js}\\
    &\lesssim 2^{\frac12\jmin}(n/\log(n))^{-\frac12} +
    2^{-\jmin\left(v - \frac12\right)}\delta_n + 2^{-\jmax s}\\
    &\lesssim \left(n/\log(n)^{1+2/pr'}\right)^{-\alpha(r')},
  \end{align*}
  which is \(O(c_n),\) as in this case \(r' > \frac1p.\)

  If instead \(r' > s',\) then \(v < \frac12.\) Pick \(j_n \in \N\) so that
  \[2^{j_n} \sim (n/\log(n)^3)^{1/(1-ptu)(2s'+1)},\]
  and we obtain
  \begin{align*}
    \norm{\hat f_n - f}_{I,\infty} &\lesssim \sum_{j=j_0}^{\jmin-1}
    2^{\frac12j}(n/\log(n))^{-\frac12} + \sum_{j=\jmin}^{j_n-1} 2^{j\left(\frac12-v\right)}\delta_n
    + \sum_{j=j_n}^\infty 2^{-js}\\
    &\lesssim 2^{\frac12\jmin}(n/\log(n))^{-\frac12} +
    2^{j_n\left(\frac12 - v\right)}\delta_n + 2^{-j_ns}\\
    &\lesssim \left(n/\log(n)^3\right)^{-\alpha(s')}.
  \end{align*}
  Similarly, if \(r' = s',\) then \(v = \frac12,\) and by the same
  calculation we obtain
  \[\norm{\hat f_n - f}_{I,\infty} \lesssim
  \left(n/\log(n)^3\right)^{-\alpha(s')}\log(n).\] As these bounds are all
  uniform over \(f \in \mathcal F,\) the result follows.
\end{proof}

We next show that the conditions of the deterministic noise model are
satisfied with probability tending to one, uniformly over functions
\(f\) in H\"older classes \(C^{\frac12}(M).\)

\begin{lemma}
  \label{lem:prob-model}
  In the probabilistic noise model, let the design points \(x_n\) be
  chosen by \autoref{alg:sas}. Then at times \(n=n_m,\) there exist
  events \(E_m\) on which condition \eqref{eq:coeff-bounds} holds with
  \(\P(E_m) \to 1;\) this convergence is uniform over \(f\) in classes
  \(C^{\frac12}(M),\) for any \(M > 0.\)
\end{lemma}

\begin{proof}
  We consider only the estimated coefficients \(\hat \beta_{j,k};\)
  the result for the \(\hat \alpha_{j_0,k}\) follows similarly. We
  will show there exist events \(E_m\) on which all possible estimates
  \(\hat \beta_{j,k}^i,\) for all possible choices of design
  \(\xi_n,\) satisfy \eqref{eq:coeff-bounds} simultaneously.

  Note that at time \(n,\) we estimate coefficients only up to level
  \(\jmax.\) For \(j_0 \le j < \jmax,\) the quantities \(i_n(j,k)\)
  must, by definition, be bounded above by a term \(\imax,\) with
  \(2^\imax \coloneqq n^2.\) Likewise, for \(n\) large, by
  \autoref{lem:discrepancy-bounded} the effective densities \(q_{l,m}
  \gtrsim 1,\) so the quantities \(i_n(j,k)\) are bounded below by
  some \(\imin,\) with \(2^\imin \gtrsim n.\)

  For \(k \in \Z,\ 0 \le k < 2^\imax,\) generate observations
  \[Y_k \coloneqq f(k2^{-\imax}) + \varepsilon_k, \quad \varepsilon_k
  \iid N(0, \sigma^2),\] and define estimates \(\hat \beta_{j,k}^i\)
  in terms of the \(Y_k,\) as in \eqref{eq:beta-hats-defn}. For any
  choice of design points \(x_n,\) our estimated coefficients \(\hat
  \beta_{j,k}\) will be distributed as \(\hat
  \beta^{i_n(j,k)}_{j,k},\) for quantities \(i_n(j, k) \in \{\imin,
  \dots, \imax\},\) also depending on the \(Y_k.\)

  Since \(f \in C^{\frac12}(M),\) for \(i \in \{\imin, \dots,
  \imax\},\) we have
  \[\abs*{\alpha_{i,k} - 2^{-\frac{i}2}f(2^{-i}k)} \le \int \abs{f(x) -
    f(2^{-i}k)}\varphi_{i,k}(x)\ dx \lesssim 2^{-i},\]%
  and so the estimates
  \[\hat \alpha_{i,k}^i \sim N\left(\alpha_{i,k} + O(2^{-i}),
    \sigma^22^{-i}\right)\] as \(n \to \infty.\) Since each estimate
  \(\hat \beta_{j,k}^i = \sum_{l=0}^{2^i-1}c_l \hat \alpha_{i,l}^i,\)
  for a vector \(c_l\) with \(\norm{c_l}_2 = 1,\) by Cauchy-Schwarz
  \(\norm{c_l}_1 \le 2^{\frac{i}2},\) and we obtain
  \[\hat \beta_{j,k}^i \sim N\left(\beta_{j,k} + O\left(2^{-\frac{i}2}\right),
    \sigma^22^{-i}\right).\]

  For given \(i,\,j,\,k,\) the probability that
  \begin{equation}
    \label{eq:bijk-inaccurate}
    \abs{\hat \beta^i_{j,k} - \beta_{j,k}} > \sigma2^{-\frac{i}2}
    \sqrt{2\log(n)}
  \end{equation}
  is thus
  \[2\Phi\left(-\sqrt{2\log(n)} + O(1)\right) \lesssim
  1/n\sqrt{\log(n)},\] using the fact that \(\Phi(-x) \le \phi(x)/x\)
  for \(x > 0.\) By a simple union bound, the probability that any
  \(\hat \beta_{j,k}^i\) satisfies \eqref{eq:bijk-inaccurate} is, up
  to constants, given by
  \[\imax 2^{\jmax} /n\sqrt{\log(n)} \lesssim 1/\sqrt{\log(n)},\]
  uniformly over \(f \in C^{\frac12}(M);\) the result follows.
\end{proof}

Finally, we may combine these lemmas to prove our constructive
results.

\begin{proof}[Proof of \autoref{thm:alg-rates}]
  For \(r \ge \frac12+\frac1p,\) the Besov classes
  \(B^r_{p,\infty}(M)\) are embedded within H\"older classes
  \(C^{\frac12}(M).\) By \autoref{lem:prob-model}, the conditions of
  the deterministic noise model therefore hold at times \(n = n_m,\)
  with probability tending to 1 as \(m \to \infty.\) In the proof of
  \autoref{lem:algorithm-accurate}, we require those conditions only
  at finitely many times \(n_m, \dots, n_{m+d},\) with \(d\) bounded
  by a fixed constant; the conclusion of that lemma thus holds also
  with probability tending to 1 as \(n \to \infty.\)
\end{proof}

\begin{proof}[Proof of \autoref{thm:alg-holder-rates}]
  Given \(f \in C^s(M, I),\) for \(j\) large, and \(k\) such that
  \(S_{j,k} \cap J \ne \emptyset,\) we have \(\abs{\beta_{j,k}} \le
  M2^{-j\left( s+\frac12\right)}.\) We note that, given this
  condition, we may prove a bound
  \[\norm{\hat f_n - f}_{J,\infty} \lesssim
  (n/\log(n))^{-\alpha(s)}\] similarly to
  \eqref{eq:holder-error-bound}; the result then follows as in the
  proof of \autoref{thm:alg-rates}.
\end{proof}

\begin{proof}[Proof of \autoref{thm:holder-rates}]
  The proof of \autoref{thm:alg-holder-rates} remains valid even if we
  set \(n = n_0,\) in which case we are describing the performance of
  a standard uniform-design wavelet-thresholding estimate.
\end{proof}
